\documentclass[11pt,a4paper,reqno]{amsart}
\usepackage{graphics, graphicx, epic}
\usepackage{latexsym}
\usepackage{amsfonts} 
\usepackage{amsthm,amssymb, amsmath}
\usepackage{color}  
\usepackage[margin=5pt,font=small]{caption} 

\newtheorem{theorem}{Theorem}
\newtheorem{corollary}{Corollary}
\newtheorem{proposition}{Proposition}
\newtheorem{lemma}{Lemma}
\newtheorem{remark}{Remark}

\newtheorem{definition}{Definition}
\newtheorem{example}{Example}
\newtheorem{conjecture}{Conjecture}

\def\supp{\text{supp}}

\def\N{\mathbb N}
\def\Q{\mathbb Q}
\def\Z{\mathbb Z}
\def\R{\mathbb R}

\def\eps{\varepsilon}
\def\ie{{\em i.e.,\ }}
\def\eg{{\em e.g.\ }}
\def\mapm{\stackrel{-}{\mapsto}}
\def\mapp{\stackrel{+}{\mapsto}}
\def\orb{\mbox{orb}}
\def\per{\mbox{per}}
\def\supp{\mbox{supp}}
\def\circle{{\mathbb S}^1}

\newcommand{\EE}{\mathcal{E}}

\newcommand{\QS}{\mathbb{Q}_{\rm s}}
\newcommand{\QMAX}{\mathbb{Q}_{s}^{\rm max}}

\newcommand{\llbracket}{[ \! [}
\newcommand{\rrbracket}{] \! ]}
\def\e{\epsilon}

\begin{document}

\title{Matching in a family of piecewise affine interval maps}

\author{Henk Bruin, Carlo Carminati, Stefano Marmi, Alessandro Profeti}
\address[Henk Bruin]{Faculty of Mathematics \\
University of Vienna \\
Oskar Morgensternplatz 1, 1090 Vienna, Austria}
\email[Henk Bruin]{henk.bruin@univie.ac.at}
\address[Carlo Carminati]{Dipartimento di Matematica \\
Universit\`a di Pisa \\ Largo Bruno Pontecorvo 5, I-56127, Italy}
\email[Carlo Carminati]{carminat@dm.unipi.it}
\address[Stefano Marmi]{Scuola Normale Superiore \\ 
Piazza dei Cavalieri 7, 56126 Pisa, Italy}
\email[Stefano Marmi]{s.marmi@sns.it}
\address[Alessandro Profeti]{Scuola Normale Superiore \\ 
Piazza dei Cavalieri 7, 56126 Pisa, Italy}
\email[Alessandro Profeti]{a.profeti@sns.it}

\date{July 2017 -- compiled \today}


\maketitle

\section{Introduction}
In the setting of the dynamics of piecewise smooth circle maps, the phenomenon of ``matching'' refers to the
property that upper and lower orbits of the singularities merge with coinciding one-sided  derivatives, 
and it becomes particularly interesting when
it happens over non-trivial intervals in parameter space.
This is observed in the family of shifted $\beta$-transformations
$x \mapsto \beta x + \alpha \pmod 1$ for the orbits
of $0$ and $1$, which may be thought of as the unique jump discontinuity of a map defined on the circle.
Similarly, matching for the orbits of $\alpha$ and $\alpha-1$,
 is well-studied in the family of Nakada's $\alpha$-continued fractions (which is probably the setting
 where this phenomenon was noted for the first time)
$$
T_\alpha:[\alpha-1,\alpha] \to [\alpha-1,\alpha], \qquad 
x \mapsto \frac{1}{|x|} - \lfloor \frac{1}{|x|} + 1-\alpha \rfloor,
$$
see \cite{N,NN,KSS,CT3}. Other families of $\alpha$-continued fractions
displaying matching were studied recently\footnote{depending on authors, matching is sometimes referred as 
"cycle property" or "synchronization".} \cite{KU1, KU2} (see also \cite{CIT}), but also in
\cite{CKS}, where a case where the underlying group is not the modular group is considered.
 
The structure of the matching set in parameter space
is often connected to various number-theoretic properties and 
bifurcation properties of the (complex) logistic family
$z \mapsto z^2+c$, see \cite{BCIT,CT1,CT2,CT3,T}.
(See also Dajani \& Kalle \cite{DK}.) 

Matching is the cause that these maps have piecewise smooth
(or even piecewise constant) invariant densities, very much like
the situation when a Markov partition would have existed.
Also, entropy depends monotonically on the parameter in the interior of each component of the matching set.

Both phenomena were observed by Botella-Soler et al.\ \cite{BOR,BORG}
in the family of piecewise 
affine maps $(G_\beta)_{\beta \in \R}$ defined by
\begin{equation}\label{eq:mapBORG}
G_\beta(x) = \begin{cases}
G_\beta^-(x) = x + 2 & \text{ if } x \le 0,\\
G_\beta^+(x) = \beta -sx & \text{ if } x \ge 0.
\end{cases}
 \ \ (s>1)\end{equation}
when the slope $s$ of the expanding branch attains some specific values ($s=2$, $s=\frac{\sqrt{5}+1}{2}$, etc.).

For us, it is more convenient to change coordinates: for fixed $s>1$ we set 
\begin{equation}\label{eq:map}
Q_\gamma(x) = \begin{cases}
x+1,  & x \leq \gamma,\\
1+s(1-x), & x>\gamma,
\end{cases}
\end{equation}
see Figure~\ref{fig:maps}.

\begin{figure}[h]
\begin{center}
\unitlength=4mm
\begin{picture}(28,10)(0,0)
\put(0,5){\line(1,0){10}} \put(5,0){\line(0,1){10}}
\put(20,0){\line(1,1){10}}
\put(0,0){\line(1,1){10}}
\put(4.7,7){\line(1,0){0.6}}\put(4.3,6.8){\tiny $2$}
\put(4.7,9.5){\line(1,0){0.6}}\put(4.3,9.3){\tiny $\beta$}
\thicklines
\put(0,2){\line(1,1){5}} \put(5,9.5){\line(1,-2){4.5}}
\put(0,2.1){\line(1,1){5}} \put(5,9.6){\line(1,-2){4.5}}
\thinlines
\put(20,2){\line(1,0){10}} \put(22,0){\line(0,1){10}}
\put(26.7,1.7){\line(0,1){0.6}}\put(26.4,0.8){\tiny $1$}
\put(24.7,1.7){\line(0,1){0.6}}\put(24.4,0.8){\tiny $\gamma$}
\put(21.7,6.7){\line(1,0){0.6}}\put(21,6.5){\tiny $1$}
\thicklines
\put(17.2,1.9){\line(1,1){7.2}} \put(24.7,10.3){\line(1,-2){5}} 
\put(17.2,2){\line(1,1){7.2}} \put(24.7,10.4){\line(1,-2){5}} 
\end{picture}
\caption{The conjugate families $G_\beta$ and $Q_\gamma$ for slope $s=2$.}
\label{fig:maps}
\end{center}
\end{figure}
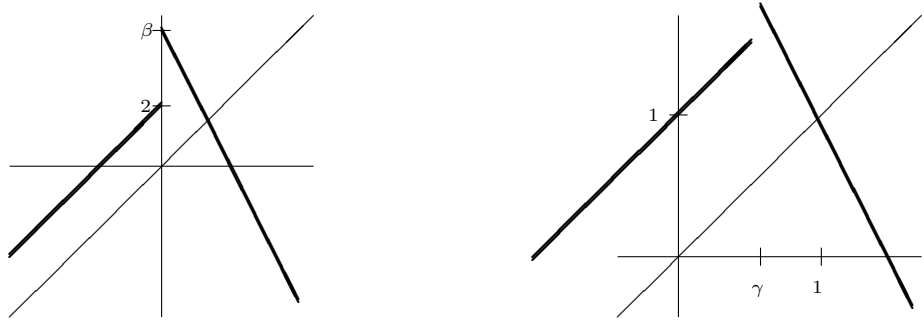

These families are conjugate via a linear change of coordinates
$$ 
H \circ Q_\gamma = G_\beta \circ H \ \ \mbox{ for } H(x):=2(x-\gamma),
\ \beta:=2(1+s)(1-\gamma).
$$
Note that this change of coordinates $\beta=\beta(\gamma)$ reverses the orientation in
 parameter space; therefore, when passing from the family $(Q_\gamma)$ to the family $(G_\beta)$, the results about
the monotonicity of entropy (see Theorem~\ref{thm:monotone})
will reverse accordingly. 

The matching in the family $(Q_\gamma)$ is  between the upper and lower orbit of the discontinuity point $\gamma$.

\begin{figure}[h]
\centering
\includegraphics[scale=0.45]{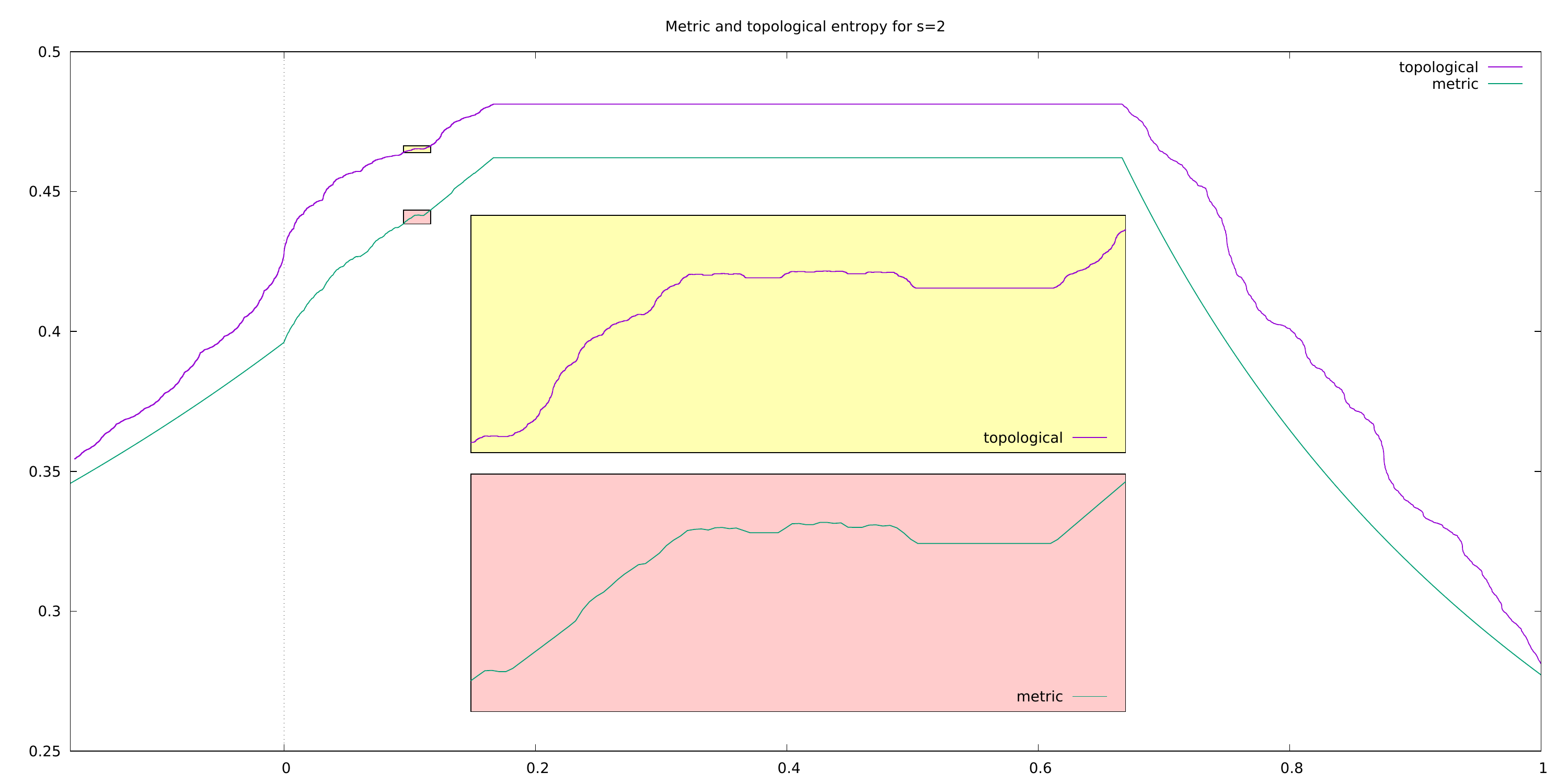}
\caption{Topological and metric entropies of $Q_\gamma$ for $s=2$ as functions of $\gamma$.}
\label{fig:entropies}
\end{figure}

As one  can guess from Figure~\ref{fig:entropies}, there are specific values of the slope $s$ for which matching
occurs for most parameters. 
Another interesting feature, evident from Figure~\ref{fig:entropies}, are the intervals on which entropy is constant,
called {\em plateaux}. We will see that this is due to a property we call {\em neutral matching},
and plateaux appear to be the closure of countably many 
components of the matching set (matching intervals), and they
consist of much more than single (or even single ``cascades'' of) 
matching intervals. 

All these features are best described in the case that the slope $s \geq 2$ is an integer.
Indeed, in all these cases we can actually prove that 
the matching occurs for Lebesgue-a.e. value $\gamma$, and the set $\EE$ where matching fails (called the {\em bifurcation set}) has zero measure (and yet the Hausdorff dimension of $\EE$ equals 1). As it is often the case in dynamical problems, the bifurcation set has a fractal structure, and
understanding the structure (and self-similarities) of this set will be important in order to characterize 
the plateaux of the entropy function. To accomplish this goal we will use methods from symbolic dynamics, 
including substitutions.

As we mentioned before, matching occurs in various one dimensional dynamical systems.
The purpose of this paper is to investigate extensively all the features connected to the 
matching phenomenon in the most simple setting, namely piecewise affine maps.

We start with a result, which applies quite in general, showing that matching  for a piecewise affine expanding map $T$ implies that the density of any 
 invariant measure is a simple function (\ie it is piecewise constant). Then we will consider the phenomenon of matching in connection  with the family
$Q_\gamma$; in this setting we shall both analyze its effects on the shape of the entropy, but we will also investigate for which slopes $s$ matching occurs, and how it is frequent. The most complete picture will be given in the case the slope $s \ge 2$ is an integer, however numerical experiments show that also  the cases of algebraic non-integer slopes such as the golden mean or other Pisot numbers are particularly intriguing and deserve further investigation.

\subsection{Piecewise constant densities}
 Let us recall that if $f$ is a piecewise continuous map the upper/lower orbit of a point $c$ are respectively
$$ f^k(c^+):=\lim_{x\to c^+} f^k (x), \ \ \  f^k(c^-):=\lim_{x\to c^-} f^k (x), \ \ \  (k\in \N)$$
The fact that the invariant density is piecewise constant
is a consequence of the fact that upper and lower orbits of each singularity eventually match with coinciding derivatives:

\begin{definition}\label{def:matchingg}
Let $T: \circle \to \circle$ be a piecewise smooth,
eventually expanding circle map; we say that $T$ satisfies  
{\em matching condition} if for every
discontinuity point $c$ (either of $T$ or of  $T'$) 
there exist positive integers $\kappa^\pm$ such that
the following  holds:
\begin{equation}\label{eq:match}
T^{\kappa^-}(c^-)=  T^{\kappa^+}(c^+), \ \ \mbox{ and } \ \  (T^{\kappa^-})'(c^-)=  (T^{\kappa^+})'(c^+).
\end{equation}
The integers $\kappa^\pm$ are called {\em matching exponents} of the discontinuity point $c$.
\end{definition}

For these maps the role of Markov partition is namely taken over by the
{\em prematching partition}, \ie the complementary intervals of
the {\em prematching set} 
\begin{equation}\label{eq:prematching}
PM:=\bigcup_c \left( \bigcup_{j=1}^{\kappa^-_c-1} T^j(c^-) \cup \bigcup_{j=1}^{\kappa^+_c-1} T^j(c^+)\right).
\end{equation}

\begin{theorem}\label{thm:density}
Let $T: \circle \to \circle$ be a piecewise affine,
eventually expanding circle map, such that 
the matching condition holds.
Then $T$ preserves an absolutely continuous invariant probability 
$d\mu = h\, dx$
and $h$ is constant on elements of the prematching partition.
\end{theorem}
The proof of Theorem~\ref{thm:density}will be given in Section~\ref{sec:2}.

\begin{remark}
By piecewise affine we mean that there is a {\bf finite} partition
on which $T$ is affine on each partition element.
For {\bf countably} piecewise affine expanding maps, the existence
of an acip is not guaranteed (cf.\ \cite{LY}), yet in many cases 
(\eg with a finite image partition) the existence of an  acip can still be proven.
If $T$ is not piecewise affine, but just piecewise smooth (such as
is the case in the $\alpha$-continued fractions, see \cite{N}), 
then the same argument implies that the density $h$ is smooth on each 
element of the matching partition.
The $\alpha$-continued fractions have a dense set of (pre)periodic points
(a property used in our proof), as a consequence of density of full cylinders.
\end{remark}

\subsection{Matching and monotonicity of entropy}
Let us now consider the family $Q_\gamma$ for some fixed $s>1$. It can be shown that for each $\gamma \in \R$ the map $Q_\gamma$ admits a unique (hence ergodic) probability measure $\mu_\gamma$ that is absolutely continuous with respect to Lebesgue (for the existence of this {\em acip}, see Lemma~\ref{lem:acip}). Therefore it is natural to study the {\em metric entropy function} $\gamma \mapsto h(\gamma):=h_{\mu_\gamma}(Q_\gamma)$, and it turns out that matching property leads to monotonicity of the  entropy function $\gamma \mapsto h(\gamma)$.

In this framework the matching condition is just
$Q_\gamma^{\kappa^-}(\gamma^-) = Q_\gamma^{\kappa^+}(\gamma^+)$ with coinciding 
one-sided derivatives (for suitable positive integers $\kappa^\pm$); the prematching set is
$$
PM_\gamma:=  \bigcup_{j=1}^{\kappa^- -1} T^j(\gamma^-) \cup \bigcup_{j=1}^{\kappa^+ -1} T^j(\gamma^+).
$$

The set of $\gamma$ for which $Q_\gamma$ satisfies the matching condition \eqref{eq:match} is possibly empty: indeed a necessary condition for \eqref{eq:match} to hold
is that $s$ must be an algebraic integer (see Theorem~\ref{thm:algebraic}). However, as soon as the slope $s$ is compatible with the phenomenon of matching, a mild additional condition guarantees that the matching is stable:

\begin{proposition}\label{prop:stablematching}
If $\gamma_0\notin PM_{\gamma_0}$ is such that
$Q_{\gamma_0}$ satisfies the matching condition \eqref{eq:match},
then for all $\gamma$ in an open neighborhood of $\gamma_0$, the map $Q_\gamma$ satisfies the 
same matching condition as $Q_{\gamma_0}$.
\end{proposition}

\begin{proof}
The branches $Q_{\gamma_0}^{\kappa^{-}}$ and $Q_{\gamma_0}^{\kappa^{+}}$ on 
either side of the discontinuity have exactly the same affine form, and since
$\gamma_0\notin PM_{\gamma_0}$,
this form will not change as $\gamma$ moves in a small neighborhood of $\gamma_0$.
\end{proof}

\begin{definition}\label{def:matching-int}
We call {\em matching set} the set of parameter values which satisfy the hypotheses of 
Proposition~\ref{prop:stablematching}. The matching set is open,  its complement is denoted by $\EE$ 
and it is called {\em bifurcation set}.
The connected components of the matching set  will be called {\em matching intervals}. 

All points belonging to the same matching interval $J$ satisfy the same matching condition, 
in particular the matching exponents $\kappa^\pm$ are the same for all $\gamma \in J$; 
the difference $\Delta := \kappa^+-\kappa^-$
is called the {\em matching index} of the matching interval $J$. 
If on some matching interval $J$ we have that $\Delta = 0$, then $J$ is called a {\em neutral matching interval}.
\end{definition}

\begin{remark}\label{rem:stable-match} 
It is also interesting to note that when $\gamma$ approaches a point 
$\gamma_*$ belonging to the boundary of a matching interval: the condition that $\gamma \notin PM_\gamma$
fails for $Q_{\gamma_*}$. However, the prematching partition turns into a Markov partition 
when $\gamma={\gamma_*}$.
\end{remark}

\quad\begin{minipage}[h]{0.95\textwidth}{\footnotesize
\begin{example}\label{ex1}
The following examples show that, for $s\geq 2$ integer, matching occurs
on some intervals:
\\
(1) $\gamma < 0$. 
$$
\begin{cases}
\gamma \mapm 1+\gamma \mapp 1-s\gamma \mapp 1+s^2\gamma, \\
\gamma \mapp 1+s-s\gamma \mapp 1-s^2+s^2\gamma \underbrace{\mapm \dots \mapm}_{s^2\ \text{\tiny times}} 1+s^2\gamma,
\end{cases}
$$
with metric entropy $h_\mu(Q_\gamma) = \frac{4\log s}{s^2+3-2\gamma(s^2-1)}$,
see Proposition~\ref{prop:mono}.
\\
(2) $\gamma > \frac{s}{s+1}$.
$$
\begin{cases}
\gamma \mapm 1+\gamma \mapp 1-s\gamma \underbrace{\mapm \dots \mapm}_{s\ \text{\tiny times}} 1+s-s\gamma, \\
\gamma \mapp 1+s-s\gamma,
\end{cases}
$$
with $h_\mu(Q_\gamma) = \frac{2\log s}{2\gamma(s+1)-s+1}$,
see Proposition~\ref{prop:mono}.
\\
(3) $\gamma \in (\frac{s-1}{s+1}, \frac{s}{s+1})$. (Neutral matching)
$$
\begin{cases}
\gamma \mapm 1+\gamma \mapp 1-s\gamma \underbrace{\mapm \dots \mapm}_{s-1\ \text{\tiny times}} s-s\gamma \mapp 1+s-s^2+s^2\gamma, \\
\gamma \mapp 1+s-s\gamma \mapp 1-s+s^2\gamma 
\underbrace{\mapm \dots \mapm}_{s\ \text{\tiny times}}  1+s-s^2+s^2\gamma.
\end{cases}
$$
with  $h_{top}(Q_\gamma) = \log(\lambda_*)$ where $\lambda_*$ is the leading root of 
$p(x)=x^s-(x^{s-1}+ x^{s-2}+...+x+1)$ (so $h_{top}(Q_\gamma) = \log \frac{1+\sqrt{5}}{2}$
if $s=2$), see Example~\ref{ex:M} below.
The metric entropy $h_\mu(Q_\gamma) = \frac2{s+1}\log s$,
see Proposition~\ref{prop:mono} combined with the fact that $\frac{s}{s+1}$ is 
the right endpoint of the top plateau $M_s = [\frac{s}{s+1}-\frac{1}{s},\frac{s}{s+1}]$,
see Theorem~\ref{thm:top}.
\\
(4) $\gamma \in (\frac{s-2s^2+s^3}{1+s^3},\ \frac{s-1}{s+1})
= (\frac{s^2-s}{s^2-s+1}\frac{s-1}{s+1},\ \frac{s-1}{s+1})$.  (Neutral matching)
$$
\begin{cases}
\gamma \mapm 1+\gamma \mapp 1-s\gamma \underbrace{\mapm \dots \mapm}_{s-2\ \text{\tiny times}} 1+2s-s^2+s^2\gamma 
\mapp 1-2s^2+s^3-s^3\gamma 
\underbrace{\mapm \dots \mapm}_{s\ \text{\tiny times}} 1+s-2s^2+s^3-s^3\gamma, \\
\gamma \mapp 1+s-s\gamma \mapp 1-s^2+s^2\gamma 
\underbrace{\mapm \dots \mapm}_{2s-1\ \text{\tiny times}}
2s-s^2 + s^2\gamma \mapm 1+s-2s^2+s^3-s^3\gamma.
\end{cases}
$$
The topological and metric entropies are as in the previous case, because
$(\frac{s-2s^2+s^3}{1+s^3},\ \frac{s-1}{s+1}) \subset M_s$.
\end{example}}
\end{minipage}

One can note that the metric as well as the topological
entropy of $Q_\gamma$ is monotone on regions where matching takes place
 (and constant in the case of neutral matching), see 
Figures~\ref{fig:entropies}. This is indeed the general case, as the next theorem shows
(see Section~\ref{sec:3} for its proof).

We prove this in the following theorem:

\begin{theorem}\label{thm:monotone}
If $\gamma$ is in a matching interval\footnote{and here we really need $\gamma \notin PM(\gamma)$} with matching index 
$\Delta = \kappa^+ - \kappa^-$, then the entropies
$$
h_\mu(Q_\gamma) 
\text{ and } h_{top}(Q_\gamma) 
\text{ are }
\begin{cases}
\text{ increasing} & \text{  if } \Delta > 0; \\
\text{ \ \ constant} & \text{  if } \Delta = 0; \\
\text{ decreasing} & \text{  if } \Delta < 0, 
\end{cases}
$$
as function of $\gamma$.
\end{theorem}

As shown in Remark~\ref{rem:smooth}, the metric entropy is smooth on matching intervals, 
and when it is increasing or decreasing then it is also  strictly increasing or strictly decreasing.
For the topological entropy we conjecture the same is true, but our proof does not go that far.
Cosper \& Misiurewicz used a different method for
piecewise affine maps similar to $Q_\gamma$ to prove that entropy
is locally constant at many parameter values \cite{CM}.

An interesting fact, to which we come back later, is that the entropy appears to stay constant
on the parameter intervals which are covered (up to a zero measure set) by a countable union of neutral matching intervals. For instance,
 if the slope $s$ is integer, the entropy is constant
on the parameter interval $M_s := [\frac{s}{s+1}-\frac{1}{s},\frac{s}{s+1}]$ (called the top plateau,
see Theorem~\ref{thm:top}), even though the
intersection of $M_s$  with the bifurcation set $\EE$ has positive Hausdorff dimension.

\subsection{Occurrence of matching in the family $(Q_\gamma)_\gamma$.}
We have seen that matching has interesting consequences on the dynamics of a system.
It is not hard  to see (numerically) that various algebraic values
of the slope $s$ frequently lead to matching.
This suggest the following natural (but intriguing) question:
\begin{quote}
For which values of the slope $s>1$ does matching occur in the 
family \eqref{eq:map}, and what portion of parameter space is covered by matching intervals?
\end{quote}

We can give the following necessary condition:
\begin{theorem}\label{thm:algebraic}
Let $s>1$ be a fixed slope, and assume that $Q_\gamma$ satisfies the matching condition. Then $s$ is an algebraic integer (\ie there is a {\bf monic} polynomial $p\in\Z[x]$ such  that $p(s)=0$).
In particular matching cannot hold for  any non-integer rational value $s$. 
\end{theorem}

\begin{proof}
Let $R$ denote the first return map to $[\gamma, +\infty)$ and let $d_j:=R^j(\gamma+1)-R^j(\gamma^+)$; it is easy to check that
$$ 
d_0=1, \ \ \ d_{j+1}=-sd_j+k_j \ \ \mbox{with } k_j\in \Z. 
$$
By induction $d_j=p_j(s)$ where $p_j\in\Z[x]$ is some monic polynomial.
If matching holds $R^m(\gamma^+)=R^m(\gamma+1)$ for some $m \in \N$
where the matching of derivatives ensures that $m$ is the same in the left and 
right hand side of this equality.
Hence $p_m(s)=0$, which means $s$ is an algebraic integer. Moreover if $s$ is rational 
then (by the Eisenstein criterion) it must be an integer.
\end{proof}

Even if we are not able to give a complete characterization of the slopes for which matching holds, we can prove 
that for integer values of $s$ matching occurs, and is prevalent (see 
Section~\ref{sec:prevalence}).

We can provide numerical evidence that matching holds and is prevalent also for other algebraic values such as the golden number $\phi:=(\sqrt{5}+1)/2$ or the plastic constant (which is the real root of $x^3=x+1$).
We will also provide some example of quadratic surd for which matching occurs but is not prevalent (see the final section).

\subsection{Prevalence of matching}\label{sec:prevalence}

In order to understand the structure of the bifurcation set $\EE$ let us note that:
\begin{enumerate}
\item For $\gamma \leq 0$, the first return map (relative to  $Q_\gamma$) to the 
interval $[\gamma, \gamma+1]$  is conjugate to $y \mapsto s^2 y + \gamma(s^2-1)\pmod 1$.
\item For $\gamma \geq s/(s+1)$ the first return map (relative to  $Q_\gamma$) to the 
interval $[\gamma, \gamma+1]$  is conjugate to $y \mapsto -s y + \gamma(s+1) \pmod 1$.
\item If $2 \le s \in \N$, then matching holds for all
$\gamma \in (-\infty, 0) \cup (s/(s+1),\infty)$ (see also Example~\ref{ex1} (1) and (2) above).
\end{enumerate}

\begin{remark}\label{rem:btrans}This means that   for general values of $s$ the structure of 
the bifurcation set outside the closed interval 
$[0,s/(s+1)]$ displays a periodic structure, and its structure inside each period can be studied 
analyzing the problem of matching for a family of generalized $\beta$-transformation with 
negative or positive slope. For the case of positive slope, such a study  has been carried over in 
\cite{BCK}, and from the results contained there it follows  that the intersection 
$\EE \cap (-\infty,0)$ has measure zero for various algebraic values of $s$, including $\phi:=(\sqrt{5}+1)/2$, 
and also all other quadratic irrational of Pisot type.
\end{remark}

For integer values of $s$, the bifurcation set $\EE$ is contained in the interval $[0,s/(s+1)]$. 
Moreover in this case we can prove that matching is typical.

\begin{theorem}\label{thm:typical}
For $2 \le s \in \N$, the bifurcation set $\EE\subset [0,s/(s+1)]$ has zero Lebesgue measure
and Hausdorff dimension $HD(\EE) = 1$, although $HD(\EE \setminus [0,\delta)) < 1$ for all $\delta > 0$.
\end{theorem}
The proof of Theorem~\ref{thm:typical} will be given in Section~\ref{ssec:typical}.

The same kind of result seems to be true for other values of the slope $s$ (such as $s=(\sqrt{5}+1)/2$ 
and other Pisot numbers) but we are still missing a rigorous proof  of prevalence in these cases. 

\subsection{Pseudocenters}
In what follows we restrict to the integer case, namely $2 \le s \in \N$; indeed in this setting one can provide a
detailed explicit description of the fractal structure of the bifurcation set.

Let us call the components of $[0, \frac{s}{s+1}] \setminus \EE$
the {\em matching intervals}.
We will see that every matching interval contains a unique
$s$-adic rational of lowest denominator. We call this point the
{\em pseudocenter}. Knowing the pseudocenter $\xi$ and its
{\bf even} $s$-adic expansion $\xi = .w$, we can reconstruct
the matching interval:

\begin{theorem}\label{T:pc}
Let $\xi$ be the pseudocenter of a matching interval $V$,
with shortest {\bf even}
$s$-adic expansion $\xi = .w$.
Then the boundary points $\partial V = \{ \xi_L, \xi_R\}$
can be obtained from $\xi$:
$$
\begin{cases}
\xi_L := .\overline{\check v v }, \\
\xi_R := .\overline{w},
\end{cases}
$$
where $v$ is the shortest {\bf odd} $s$-adic expansion of
$1-\xi$, and $\check v$ is the bit-wise negation of $v$
(\ie $\check a = s-1-a$ for $a \in \{ 0, \dots, s-1\}$).
Furthermore, $\check w < v$ and $\check v < w$.
\end{theorem}
The proof of Theorem~\ref{T:pc} is obtained combining Lemma~\ref{L:ixi} and Proposition~\ref{P:maximal} of Section~\ref{sec:pseudo}.
According to Theorem~\ref{thm:monotone}, the
matching index $\Delta = \Delta_\xi$ determines how the entropy
depends on the parameter $\gamma \in V$. On the other hand, the following Theorem (proved in Section~\ref{sec:mi}) shows that the 
matching index $\Delta_\xi$ can be obtained from the $s$-expansion of $\xi$ using a simple recipe:

\begin{theorem}\label{thm:matching_index}
For a word $w \in \{ 0, 1, \dots, s-1\}^*$, define
$$
\| w \| = \sum_{j=0}^{s-1} (s-1-2j ) |w|_j
\quad \text{ for } |w|_j = \#\{ i \in \{ 1, \dots, |w|\} : w_i = j\}.
$$
Every $s$-adic pseudocenter $\xi$ with even expansion $\xi = .w$,
has matching index 
$$
\Delta_\xi = \frac{s+1}{2} \| w \|,
$$
which is always multiple of $s+1$.
\end{theorem}

As a consequence of this description, we will prove that on the left of each matching interval $V$ there  
exists a period doubling cascade of adjacent neutral matching
intervals (see Proposition~\ref{P:period-doubling}).

\subsection{Plateaux and self-similarities}
As mentioned before there are intervals in parameter
spaces, called {\em plateaux}, where the entropy is constant.
In the previous section we saw that, for $2\leq s \in \N$, there is neutral 
matching in every period doubling cascade, but plateaux always represent 
more than a single periodic doubling  cascade. In fact we can give an explicit description of the highest plateau
in the integer slope case:

\begin{theorem}\label{thm:top}
For fixed integer slope $s \ge 2$, the family $Q_\gamma$ has {\em top plateau}
\begin{equation}\label{eq:M}
M_s := [\frac{s}{s+1}-\frac1s, \frac{s}{s+1}] 
\end{equation}
where both the metric and the topological  entropy are constant.
This plateau consists of more than a single period doubling cascade: in fact the intersection
$M_s \cap \EE$ has positive Hausdorff dimension.
\end{theorem}

For instance for the particular family corresponding to the slope $s=2$ we get that both 
metric and topological entropy 
of $Q_\gamma$ are constant for $\gamma \in [1/6, 2/3]$; this implies that 
 the entropy of the maps $G_\beta$  is constant for $\beta \in [2,5]$, thus proving a conjecture 
 which was stated in \cite{BORG}.

As a matter of fact Theorem~\ref{thm:top} is just a particular case of Theorem~\ref{thm:monotone3} which shows that the entropy is constant on several other intervals: in Section~\ref{sec:tuning} we shall give a detailed description of the plateaux. In order to do this we are led to study some self-similar features of the graph of the entropy function which have also other interesting consequences.
\medskip

\section{Piecewise constant density}\label{sec:2}

In this section, we prove Theorem~\ref{thm:density}, and then give some
general discussion of how to find the invariant density,
based on the fact that the prematching partition plays the role of a Markov partition.

We start with a simple lemma.

\begin{lemma}\label{lem:dense_preperiodic}
Let  $T:\circle \to \circle $ be a piecewise linear, eventually expanding map with a 
finite number of discontinuities. Then the preperiodic points of $T$ are dense.
\end{lemma}

\begin{proof}
Take $m\in \N$ such that $S:=T^m$ satisfies $|S'(x)|>3$ almost everywhere. 
Let $C$ be the set of discontinuities of $S$.
We shall prove that preperiodic points of  $S$ are dense, and this will imply our claim. 
Since every open interval will eventually be mapped onto a neighborhood of some $c\in C$, it is enough to prove that we can find preperiodic points arbitrarily close to any discontinuity. Indeed, let us fix $\epsilon>0$ such that the family of intervals
$F:=\{ (c-\epsilon, c), (c,c+\epsilon) \ : \ c\in C \}$ are all disjoint.
We iterate any one of these intervals; its size will grow geometrically until it eventually covers some $c \in C$, and it will cover a whole 
element of $F$. We repeat the argument iterating this new element, and since $F$ has finite cardinality, we get that the original element will map on another element which eventually maps onto itself. This leads to a preperiodic point.
\end{proof}

\begin{proof}[Proof of Theorem~\ref{thm:density}]
Since $T$ is eventually expanding, $T$ preserves an absolutely continuous
probability.

Let $\kappa^\pm_c$ be the matching indices of the discontinuity points of $T$,
so the prematching set is $\bigcup_c (\bigcup_{j=1}^{\kappa^-_c-1} T^j(c_-) \cup
\bigcup_{j=1}^{\kappa^+_c-1} T^j(c_+)$. Let $x \in \circle$ be any other point.
If $x$ is wandering, then $h$ is constantly zero near $x$, so
we can restrict our attention to the non-wandering set. The fact that $T$ is expanding
implies that there are preperiodic points on either side of $x$, so there is 
a open interval $J \owns x$, disjoint from the prematching set,
and {\em nice} in the sense that $\orb(\partial J) \cap J = \emptyset$.
By the previous lemma, there is an abundance of (pre)periodic points
to choose from for $\partial J$. 
Let $R$ be the first return map to $J$.
If $R:K \to J$ is some  branch of this return map, which is not
onto, then $\partial R(K)$ contains a point in $\orb(\partial J)$ (but this is impossible, 
because $J$ is nice), or a point $p$ in the post-critical set.
In the latter case, there is $0 < j < \tau(K)$ (where $\tau(K)$ is the first 
return time of $K$ to $J$) such that some critical point $c \in \partial T^j(K)$.
Since $J$ is disjoint from the prematching set, 
$\tau(K) - j \ge \max\{ \kappa^-_c, \kappa^+_c\}$.
Furthermore, there is an interval $K'$ adjacent to $K$ such that $T^j(K')$ 
and $T^j(K)$ have $c$ as common boundary point (say they lie to the left resp.\ right of $c$).
Due to the matching, $\overline{ T^{j+\kappa^-_c}(K') \cup T^{j+\kappa^+_c}(K) }$
contains a two-sided neighborhood of $p$ and since the matching is strong,
the derivative is constant on this entire interval.

It follows that $R$ has only affine, surjective branches, so 
$R$ preserves Lebesgue measure $m$. On the other hand, since $R$ is a first return map, $\mu|_J$ 
is also $R$-invariant, so $\mu|_J$ coincides with Lebesgue 
measure, up to a constant. This constant is in fact the density $\rho|_J$.
Since $x$ was arbitrary, it follows that $h$ is constant away from the prematching set.
\end{proof}

Analyzing the above proof we may note that actually it would be enough to require the matching 
condition only for those singularities which belong to the support of the invariant measure 
(in which case the prematching set should be redefined accordingly).

Let us consider any finite partition $\{P_i\}_{i=1}^N$ of $\circle$ which is finer than the partition 
generated by the singular points of $T$. Then $T$ is affine on each $P_i$, $|DT(x)|=t_i$ 
for all $x\in P_i$ and setting $\Pi_{i,j} = |T(P_i) \cap P_j|/|P_j|$ we get, 
$$|P_i|=\frac{1}{t_i} |TP_i|=\sum_{j=1}^N \frac{\Pi_{ij}}{t_i}|P_j| \ \ \ \ \ (1\leq i \leq N)$$
and thus $(|P_1|,...,|P_N|)^T$ is a right eigenvector (corresponding to eigenvalue 1) of the 
matrix $A=(A_{i,j}) = (\frac1{t_i} \Pi_{i,j})$.

\qquad\begin{minipage}[h]{0.9\textwidth}{\footnotesize 
\begin{example}
In the case $s=2$ and  $\gamma \in [-\frac13,0]$, we use the partition
$\{ [-3+4\gamma, -2+4\gamma], \ [-2+4\gamma, -1+4\gamma], \ [-1+4\gamma, 4\gamma],\
[4\gamma, \gamma], [\gamma, 1+\gamma], \ [1+\gamma, 1-2\gamma], \ [1-2\gamma, 3-2\gamma]\}$,
with transition matrix
$$
\Pi = \begin{pmatrix}
0 & 1 & 0 & 0 & 0 & 0 & 0 \\
0 & 0 & 1 & 0 & 0 & 0 & 0 \\
0 & 0 & 0 & 1 & \theta & 0 & 0 \\
0 & 0 & 0 & 0 & 1-\theta & 0  & 0 \\
0 & 0 & 0 & 0 & 0 & 0 & 1 \\
0 & 0 & 0 & 0 &  1-\theta & 1 & 0 \\
1 & 1 & 1 & 1 & \theta &  0 & 0
\end{pmatrix}
 \quad \text{ and } \quad
A = \begin{pmatrix}
0 & 1 & 0 & 0 & 0 & 0 & 0 \\
0 & 0 & 1 & 0 & 0 & 0 & 0 \\
0 & 0 & 0 & 1 & \theta & 0 & 0 \\
0 & 0 & 0 & 0 & 1-\theta & 0  & 0 \\
0 & 0 & 0 & 0 & 0 & 0 & \frac12 \\
0 & 0 & 0 & 0 &   \frac{1-\theta}{2} &  \frac12 & 0 \\
\frac12 & \frac12 & \frac12 & \frac12 & \frac{\theta}{2} &  0 & 0
\end{pmatrix},
$$
where $\theta = |T(P_3) \cap P_5|/|P_5| = 1+3\gamma$.
\end{example}}
\end{minipage}

If we make the ansatz that $T$ has an invariant measure $d\mu(x)=\rho(x) dx$ 
with density $\rho(x)=\sum r_i 1_{P_i}(x)$ (for indicator functions $1_P$), 
then the condition $\mu(P_j)=\mu(T^{-1}(P_j))$ implies that
\begin{equation}\label{eq:necessary}
r_j = \sum_i \frac{r_i}{t_i} \Pi_{i,j},  \qquad1\leq j \leq N.
\end{equation} 
This means that $(r_1,...,r_N) $ must be a left eigenvector (corresponding to eigenvalue $1$) 
of the matrix $A$ defined above. Since $1\in \sigma(A)=\sigma(A^T)$, equation \eqref{eq:necessary} 
always admits nontrivial solutions, and it is well known that if $\{P_i\}_{i=1}^N$ 
is a Markov partition for $T$ (\ie if $\Pi_{ij} \in \{0,1\}$ for all $i,j$) any such 
nontrivial (normalized) solution corresponds to an invariant (probability) measure for $T$.

Obviously the Markov property cannot hold in general.
However, if $T$ satisfies the hypotheses of Theorem~\ref{thm:density} and the 
partition $\{P_i\}_{i=1}^N$  is determined by the union of the singular set 
and prematching set then the existence of an invariant density is 
guaranteed by Theorem~\ref{thm:density} and equation
\eqref{eq:necessary} provides an effective way of computing the invariant density.

\begin{remark}\label{rem:smooth}
Even if $T$ is defined on $\R$ rather than on $\circle$, Theorem~\ref{thm:density} 
applies. Indeed, the support of the invariant measure is contained 
an invariant interval; identifying the endpoints of such an interval gives rise to a circle map to 
which Theorem~\ref{thm:density} applies.
If $T=Q_\gamma$ for some $\gamma$ varying in a matching interval and 
we consider the partition generated by the points $\{\gamma\}\cup PM_\gamma$,
then we see from \eqref{eq:necessary} that the elements of the matrix $A$  change 
smoothly as $\gamma$ varies in the matching interval, and so does also the unique normalized 
solution of equation \eqref{eq:necessary}. The entropy can be computed by 
the Rokhlin formula
$$
h(\gamma)=\int \log|DQ_\gamma(x)| d\mu_\gamma (x)= (\log s) \sum_{P_i \subset (\gamma, +\infty)} r_i |P_i|
$$
and thus it is smooth as well, see also Corollary~\ref{cor:entropysmooth}.
\end{remark}

\qquad\begin{minipage}[h]{0.9\textwidth}{\footnotesize 
\begin{example}\label{ex2}
In  the case $s=2$ and $\gamma \in [\frac23,1]$, we have the partition 
$\{ [1-2\gamma, 2-2\gamma]\ , \ [2-2\gamma, \gamma]\ , \ [\gamma,1+\gamma]\}$
and matrices
$$
\Pi = \begin{pmatrix}
0 & 1 & \theta \\
0 & 0 & 1-\theta \\
1 & 1 & \theta
\end{pmatrix}
\quad\text{ and }\quad 
A = \begin{pmatrix}
0 & 1 & \theta \\
0 & 0 & 1-\theta \\
\frac12 & \frac12 & \frac{\theta}{2},
\end{pmatrix}
$$
so in this case there are three intervals whose boundary
maps to the matching point (here $\theta = 3(1-\gamma)$).
The normalized left and right eigenvectors are, respectively
$$
\frac{1}{5-2\theta}
\begin{pmatrix}
1 & 2 & 2
\end{pmatrix}
\quad\text{ and }\quad
\begin{pmatrix}
1  \\
1-\theta \\
1
\end{pmatrix}
$$
Note that $\theta = 0$ as $\gamma = 1$,
and $\gamma = \frac23$ corresponds to $\theta = 1$.
The metric entropy of $\mu$ is
$h_\mu(Q_\gamma) = \frac{2\log 2}{6\gamma-1}$ by the Rokhlin formula.
\end{example}}
\end{minipage}

\section{Monotonicity of entropy}\label{sec:3}
From now on we shall focus on the family of maps $(Q_\gamma)_{\gamma\in \R}$ defined in the introduction. Note that
using this notation we omit the dependence on the slope: it is implicit that  different choices of the slope $s>1$ 
will give rise to  different families. We will show how the matching property  affects  
some dynamical invariants of these families, 
 such as topological or metric entropy.

Let us first note that (for every fixed value $s>1$ of the slope)
each $Q_\gamma$ admits a unique absolutely continuous
invariant probability.

\begin{lemma}\label{lem:acip}
The maps $Q_\gamma$ are Lebesgue ergodic; hence there is a unique acip $\mu_\gamma$.
\end{lemma}

\begin{proof}
Let $m \in \N$ be such that $S:=Q_\gamma^m$ is expanding; 
in the terminology of \cite{Zw}, $S$ is an AFU map, and according
to  \cite[Lemma 4]{Zw} there is a finite number of disjoint $S$-invariant
 open sets, each of them supporting an ergodic absolutely continuous measure.
However, since every point has a neighborhood that will be eventually 
mapped onto a neighborhood of the single discontinuity $\gamma$, 
there can be just one ergodic component. 
\end{proof}

We shall now consider the map $\gamma \mapsto h_\mu(\gamma):=h(Q_\gamma,\mu_\gamma)$ which associates to every parameter $\gamma$ the metric entropy of $Q_\gamma$.  Classical  general results \cite{KL} ensure that this map is H\"older continuous (of any exponent $\eta<1$); and yet on intervals where matching holds the entropy function is much more regular, in fact it is analytic
(see Remark~\ref{rem:smooth} and Corollary \ref{cor:entropysmooth}).
  
But the most evident feature displayed by the metric entropy is the fact that it is monotone on each matching interval 
(and the kind of monotonic behavior is determined by sign of the matching index). Different matching intervals are kneaded in a complex way, so that in the end the global regularity of the entropy is no better than  H\"older continuous.


\begin{proof}[Proof of Theorem~\ref{thm:monotone}]
Let $J$ be a nice neighborhood of $\gamma$ disjoint from the prematching set.
Let $R = Q_\gamma^\tau$ be the corresponding first return map; it preserves
normalized Lebesgue measure $m$. Let $J_0$ be the domain of $R$ 
containing $\gamma$.

\begin{remark}
The argument so far used that $\gamma \notin PM_\gamma$.
For example, if $\gamma = \frac{1}{s+1}$, then 
$$
\begin{cases}
\gamma^- \mapm 1+\frac{1}{s+1} \mapp \frac{1}{1+s} = \gamma^+,\\[1mm] 
\gamma^+ \mapp 1+\frac{s^2}{s+1} \mapp 2s-s^2+\frac{1}{s+1} \mapm \dots \mapm \frac{1}{1+s} = \gamma^+,
\end{cases}
$$
so the first return of $\gamma^+$ happens after the first return
of $\gamma^-$, and there is in fact no matching of derivatives
at any iterate. For this reason, we have to assume
that $\gamma$ is in a matching 
interval.
\end{remark}

Assume first that there is $N \ge 0$ such that $Q_\gamma^N(q) \in J$
for the matching point $q = Q_\gamma^{\kappa^-}(\gamma^-) =
 Q_\gamma^{\kappa^+}(\gamma^+)$.
As $\gamma$ varies within $J_0$, the orbit of $\partial J$ doesn't
change, and hence $J$ and $R$ do not change either. Also $J_0$
doesn't change, except that the discontinuity point $\gamma$ moves 
within it, so as $\gamma$ increases, the proportion of $J_0$ with return
time $\kappa^++N$ decreases whereas
the proportion of $J_0$ with return time $\kappa^-+N$ increases. 
This means that $\int_J \tau\ dm$ decreases/increases/remains
unchanged
according to whether $\Delta > 0$,  $\Delta < 0$ or $\Delta = 0$. 
Therefore, using Abramov's formula, we obtain that
\begin{equation}\label{eq:abramov}
h_\mu(Q_\gamma) = \frac{1}{\int_J \tau\ dm} h_m(R)
\end{equation}
increases/decreases/remains unchanged accordingly. 

If no such $N$ can be found, \ie $\gamma$ is not recurrent,
then we repeat the argument with some other nice interval
$J$ containing $q$. In this case, there can be several branches $J_0$
that pass through $\gamma$ (and are matched again) before
returning to $J$, but the proportions of points in such $x \in J_0$
that use the $Q_\gamma^{\kappa^+}$-branch to return decreases
as $\gamma$ increases, regardless of the branch $J_0$.
Thus the same result holds.
\end{proof}

Looking carefully equation~\eqref{eq:abramov} we can realize that it has an interesting consequence:
\begin{corollary}\label{cor:entropysmooth}
The function $\gamma \mapsto 1/h_\mu(\gamma)$ is locally affine on every matching interval. 
Therefore, if $(a,b)$ is a matching interval, the metric entropy is given by
\begin{equation}\label{eq:interpolation}
h(x)=\frac{h(b)h(a)(b-a)}{(b-x)h(b)+(x-a)h(a)}
\qquad \text{ for all } \gamma \in (a,b),
\end{equation}
where, for sake of readability, we abbreviated $h(\gamma) := h_\mu(\gamma)$.
\end{corollary}

Formula \eqref{eq:interpolation} is particularly interesting in practice: indeed, 
even if matching fails on the endpoints of a matching interval $(a,b)$, 
when $\gamma\in \{ a,b \}$ the map $Q_\gamma$ admits a Markov partition.
Hence both $h(a)$ and $h(b)$ can be computed in a standard way.

Theorem~\ref{thm:monotone} implies that the metric entropy is constant on every neutral matching interval, 
but it turns out that -in many cases- the intervals where the entropy is constant are 
clusters of countably many neutral matching intervals. We shall describe phenomenon later on, providing 
rigorous proofs in the case $2\leq s \in \N$. A peculiar feature of the families with integer slope is the 
presence of two unbounded matching intervals, where metric entropy can be explicitly computed:

\begin{proposition}\label{prop:mono}
Let $s \in \N$, $s\geq 2$.
The metric entropy of $Q_\gamma$ is 
$$
h_\mu(Q_\gamma) =\begin{cases} 
\frac{4\log s}{s^2+3-2\gamma(s^2-1)} \qquad & \text{ if } \gamma \leq 0,\\[1mm]
\frac{2\log s}{2\gamma(s+1)-s+1} & \text{ if } \gamma \geq s/(s+1).
\end{cases}
$$
\end{proposition}

\begin{proof}
First assume that $\gamma\leq 0 $.
We use the first return map $R = Q_\gamma^\tau:[\gamma,1+\gamma) \to [\gamma,1+\gamma)$
with first return time $\tau(x) = \min\{n \ge 1 : Q_\gamma^n(x) \in [\gamma,1+\gamma]\}$, i.e.
$$
R(x) = s^2x + k(x)  \qquad \text{ where } k(x)\in \Z \text{ such that } \gamma \leq s^2x+k(x) < 1+\gamma.
$$
This map has constant slope $s^2$ and preserves Lebesgue measure,
so its entropy $h_m(R) = 2\log s$. 
Using Abramov's formula $h_\mu(Q_\gamma) \int \tau \, dm = h_m(R)$.
A somewhat tedious computation gives the above answer.

For $\gamma \geq \frac{s}{s+1}$, the first return map $R$ to $[\gamma,1+\gamma)$
has slope $-s$  and metric entropy $h_m(R) = \log 2$. Again, Abramov's formula gives the required answer. 
\end{proof}

In \cite{BORG} the authors provide an argument which proves that topological entropy stays constant on neutral matching intervals. 
In fact, an analogue of Theorem~\ref{thm:monotone} seems to be true for topological entropy,  we will prove it in the  
integer slope case in the next section.

\section{Integer slopes}\label{sec:integer}
Throughout this section, $s\geq 2$ will be an integer, and $\QS = \{ p/s^m : p \in \Z, m \ge 0\}$ will denote the set of {\em $s$-adic rationals}.

When the slope $s \in \N$ we can give a quite complete account of the phenomenon of matching and related features. Many (but not all) of these features can also be observed for other values of the slope, but for these cases we still miss rigorous proofs (see Section \ref{sec:numerical} for a more detailed discussion of these issues).

Let us point out that in this integer slope case the bifurcation set is bounded\footnote{This is not true for general slope.}. Indeed,
from Example~\ref{ex1}, it follows that
matching holds on the two half lines $(-\infty,0)$ and 
$(\frac{s}{s+1},+\infty)$. Hence $\EE \subset [0, \frac{s}{s+1}]$.

\subsection{Matching is typical}\label{ssec:typical}

The main focus of this section is proving Theorem \ref{thm:typical}. 
A key ingredient to reach this goal is a neat characterization 
of the bifurcation set. Additionally, this characterization 
has some consequences that will be very useful in Section~\ref{sec:plateaux}.
 
Let us observe that, since $1$ is a fixed point for $Q_\gamma$ for all $\gamma \in \R$, we get that  if $\gamma \in \QS$; then both positive and negative orbit of $\gamma$ end 
up in $1$, and derivatives can be made to match.
Furthermore, $\gamma \notin PM_\gamma$, see \eqref{eq:prematching}.
This implies that the matching persist under a small perturbation
in $\gamma$, and hence matching is an open and dense condition.

We now claim that the first returns of $Q_\gamma$ on $[0,1)$
are  modeled by the map $g(x):=s(1-x) \bmod 1$.

\begin{lemma}\label{L:firstreturn}
 Let $x\in (0,1)$ and let $R(x)$ denote the first return of $Q_\gamma^k(x)$
 to $[0,1)$. Then
 $$R(x):=\left\{
\begin{array}{ll}
 g(x) & \mbox{ if } x\in (0,\gamma)\\
 g^2(x) & \mbox{ if } x\in (\gamma,1)
\end{array}
\right.
$$
 Note that $R$ is not defined for $x=0$ because in this case $Q_\gamma^k(x)=1$ for all $k\geq 1$ 
 (it never returns to $[0,1)$). 
\end{lemma}

\begin{figure}[h]
\begin{center}
\unitlength=3mm
\begin{picture}(16,16)(-3,-2) 
\put(0,-2){\line(0,1){16}}
\thicklines
\put(-2,-2){\line(1,0){16}} \put(-2,-2){\line(0,1){16}} 
\put(-2,14){\line(1,0){16}} \put(14,-2){\line(0,1){16}} 
\put(0,6){\line(1,4){2}} \put(2,-2){\line(1,4){4}} 
\put(6,-2){\line(1,4){4}} \put(10,-2){\line(1,4){4}} 
\put(-2,14){\line(1,-2){2}}
\put(-0.3, -2.97){\small $\gamma$}
\end{picture}
\caption{The map $R$ for slope $s = 2$ and $\gamma = 1/8$.}
\label{fig:F}
\end{center}
\end{figure}
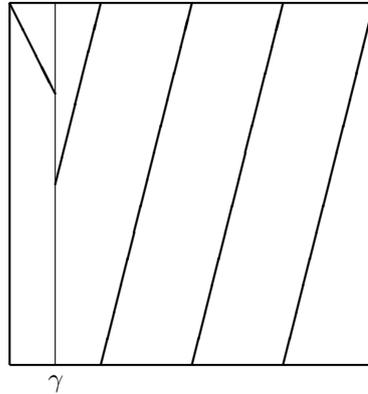

\begin{proposition}\label{P:HB}
 Let $\gamma \in [0,1]$ be fixed. Then the following conditions are equivalent:
 \begin{enumerate}
  \item[(i)]  $g^k(\gamma)<\gamma$ for some $k\in \N$;
  \item[(ii)] $\gamma$ belongs to the matching set. 
 \end{enumerate}
 In other words, the bifurcation set is 
 \begin{equation}\label{EE}
 \EE = \{ \gamma \in [0,1] \ : \ g^k(\gamma) \geq \gamma \ \forall k\in \N \}.
\end{equation}
\end{proposition}

\begin{proof}
It is immediate to check that if $\gamma=p/s^m$ then both conditions (i) and (ii) hold, so let us 
assume that $\gamma\in[0,1]$ is not of this particular form, which is the same as 
$g^k(\gamma)>0$ for all $k \in \N$.

$\mathbf{(ii) \Rightarrow (i)}$. Let us assume that $g^k(\gamma)\geq \gamma $ for all $k\in \N$. 
It follows by induction that for all integer value $j\geq 1$ there is  $\delta_j>0$ such that
$$
\begin{array}{ll}
R^j(x)=g^{2j}(x)>\gamma & \forall x\in (\gamma, \gamma+\delta_j)\\
R^j(x)=g^{2j-1}(x)>\gamma & \forall x\in (\gamma -\delta_j, \gamma)
\end{array}
$$
and this implies that the upper orbit $\{R^j(\gamma^+), \ \ j \in \N\}$ coincides with even powers of $g$, while the lower orbit $\{R^j(\gamma^-), \ \ j \in \N\}$ coincides  with odd powers of $g$, hence matching cannot take place because upper and lower orbits will never meet. 

$\mathbf{(i) \Rightarrow (ii)}$.
Let us assume that  $g^k(\gamma)<\gamma$ for some integer $k$, and let us set
$$k_0:=\min\{k\in \N \ : \ g^k(\gamma)\in (0,\gamma)\}$$
We shall split the discussion depending on whether $k_0$ is even or odd.

If $k_0=2h_0$ then
$$\begin{array}{ll}
R^j(\gamma^+)=g^{2j}(\gamma)  & 1\leq j \leq h_0,\\
R^j(\gamma^-)=g^{2j-1}(\gamma) & 1\leq j\leq h_0+1.
\end{array}
$$
Since $g^{2h_0}(\gamma)<\gamma$ we get that $ R^{h_0+1}(\gamma^+)=g^{2h_0+1}(\gamma)= R^{h_0+1}(\gamma^-)$, 
hence also derivatives match $ (R^{h_0+1})'(\gamma^+)=(-s)^{2h_0+1}= (R^{h_0+1})' (\gamma^-)$
and $\gamma \notin PM_\gamma$.

If $k_0=2h_0+1$ then
$$
\begin{array}{ll}
R^j(\gamma^+)=g^{2j}(\gamma)  & 1\leq j \leq h_0+1\\
R^j(\gamma^-)=g^{2j-1}(\gamma) & 1\leq j\leq h_0+1
\end{array}
$$
Since $g^{2h_0+1}(\gamma)<\gamma$ we get that $ R^{h_0+2}(\gamma^-)=g^{2h_0+2}(\gamma)= R^{h_0+1}(\gamma^+)$. Hence also derivatives match $ (R^{h_0+2})'(\gamma^-)=(-s)^{2h_0+2}= (R^{h_0+1})' (\gamma^+)$ and $\gamma \notin PM_\gamma$.
\end{proof}

From the proof of Proposition~\ref{P:HB} one can easily get the following corollary, which will be used later on:

\begin{corollary}\label{C:periodic} 
Let us assume that  $\gamma$ belongs to the matching set. Then
the upper orbit of $\gamma$ is periodic for $Q_\gamma$ if and only if  the lower orbit of $\gamma$ is periodic,
and these two orbits  have the same multiplier.

Let us call $p_u$ and $p_\ell$ the period of the upper and lower orbit, respectively; then 
$$ p_u \geq \kappa^+, \ \ \ p_\ell \geq \kappa^-, \ \ p_u-p_\ell =\kappa^+-\kappa^- = \Delta$$  
where $\kappa^\pm$ are the matching exponents of $\gamma$
\end{corollary}

 \begin{proof}
 Suppose the upper orbit is periodic under $Q_\gamma$ with period $p_u$, in view of the matching condition it 
 suffices to check that $p_u \geq \kappa^+$, \ie the period is not reached before matching takes place.
 
Let us use the same notation as in Proposition~\ref{P:HB}; the first returns of the upper orbit to $(0,1)$ are $R^j(\gamma^+)=g^{2j}(\gamma)>\gamma$ as far as $2j<k_0$, and two cases are possible:
\begin{enumerate}
\item[{[even]}] $k_0=2h_0$, $R^{j}(\gamma^+)>\gamma$ for $j< h_0$, $R^{h_0}(\gamma^+)<\gamma$ 
and $R^{h_0+1}(\gamma^+)=R^{h_0+1}(\gamma^-)$;\\
\item[{[odd]}] $k_0=2h_0+1$, $R^{j}(\gamma^+)>\gamma$ for $j\leq h_0$ and 
$R^{h_0+1}(\gamma^+)=R^{h_0+2}(\gamma^-)$.
\end{enumerate} 
In either case it is clear that $p_u \geq \kappa^+$, \ie the period is not completed before 
matching occurs.
Therefore we can change the upper orbit into the lower orbit just substituting the 
initial $\kappa^+$ items of the first with the initial $\kappa^-$ items of the latter; 
thus we have that the period changes accordingly $p_u-p_\ell =\kappa^+-\kappa^-$, 
and yet the multiplier does not change, because of the matching of derivatives.
\end{proof}

\begin{lemma}\label{L:pee}
Let $\gamma_0\in (0,1)$ and let $p_0>\gamma_0$ be a periodic point for $Q_{\gamma_0}$.
If $p_0\in \EE$ then  $p_0$ is the right endpoint of a matching interval. 
\end{lemma}
\begin{proof}
Let $j_0$ be the period of $p_0$, so $Q^{j_0}_{\gamma_0}(p_0)=p_0$; let us recall that all 
elements of the form $Q^{j}_{\gamma_0}(p_0)$
falling in $(0,1)$ are of the form $g^k(p_0)$ for some $k$. 
In particular, since $p_0\in \EE$ implies $g^k(p_0)\geq p_0$, there are no elements of the orbit of $p_0$ 
in the interval $[\gamma_0,p_0)$. Therefore $p_0$ is $k_0$-periodic for all $\gamma \in [\gamma_0, p_0)$, 
and its orbit does not change as $\gamma$ ranges in this interval, 
i.e., $Q^{j_0}_{\gamma}(p_0)=p_0 \ \forall \gamma \in [\gamma_0, p_0)$. 
Moreover, since $\gamma$ cannot belong to the orbit of $p_0$ we also have that $Q^{j_0}_{\gamma}=g^m$ 
in an open neighborhood of $p_0$. We now split the discussion in two cases.

{\bf m even.} Set $U_\gamma:=(\gamma, p_0)$  and choose $\delta>0$ so small that  $$\gamma\notin  Q^{j}_{\gamma}(U_\gamma) \ \ \forall k : \ 0\leq j <j_0, \ \ \forall \gamma: \ p_0-\delta<\gamma<p_0.$$
Thus $Q^{j}$ is continuous on $U_\gamma$ for all $0\leq j\leq j_0$; in particular   $Q^{j_0}(t)=g^m(t) \ \forall t \in U_\gamma$ as soon as $\gamma \in (p_0-\delta,p_0)$. Since $g^m$ is expanding, orientation preserving and $g^m(p_0)=p_0$ we get that $g^m(\gamma)<\gamma \ \  \forall \gamma: \ p_0-\delta<\gamma<p_0$, and by Proposition \ref{P:HB} this proves the claim.

{\bf m odd.}  In this case $g^m$ is orientation reversing, but  we can choose $\delta>0$ so small that setting $U_\gamma:=(\gamma, p_0)$ then $\gamma\notin  Q^{j}_{\gamma}(U_\gamma) \ \ \forall k : \ 0\leq j <2j_0, \ \ \forall \gamma: \ p_0-\delta<\gamma<p_0$, so that $Q^{j}$ is continuous on $U_\gamma$ for all $0\leq j\leq 2j_0$; in particular $Q^{2j_0}(t)=g^{2m}(t) \ \forall t \in U_\gamma$ as soon as $\gamma \in (p_0-\delta,p_0)$. Since $g^{2m}$ is expanding, orientation preserving and $g^{2m}(p_0)=p_0$
 the claim follows just as in the previous case.
\end{proof}

\begin{remark}\label{R:PP}
The last two results above are useful to describe how periodic orbits change as $\gamma$ changes.
Indeed a periodic orbit can only change when $\gamma$ crosses it, and this crossing may take place in a matching interval, in which case the periodic orbit persists but its period decreases by $\Delta$ as $\gamma$ increases, or at the right endpoint of a matching interval. In the latter case, which is quite "rare", the periodic point may even disappear.
\end{remark}

For $t\in [0,1]$ define
\begin{equation}\label{eq:kappa}
 K(t):=\{ x\in [0,1] \ : \ g^k(x)\geq t  \ \forall k\in \N \},
\end{equation}

Since $g$ is ergodic, the Lebesgue measure of $K(t)$ is zero. Moreover $\EE\cap [t,1] \subset K(t)$. 
From this, it is easy to give the

\begin{proof}[Proof of Theorem~\ref{thm:typical}]
Lebesgue measure is preserved by $g:[0,1) \to [0,1)$,
so the Ergodic Theorem implies that $\inf\{ g^k(\gamma) : k \ge 1\} = 0$
for a.e.\ $\gamma$. Proposition~\ref{P:HB} implies that each such $\gamma \notin \EE$.

Now for the statement of the Hausdorff dimension, note
that $HD(K(t)) \to 1$ as $t \to 0$.
Indeed, fix $m \in M$ and let $K' = \{ x \in [0,1) : g^{km}(x) \geq s^{-m}
 \ \forall k\in \N \}$. This is the set of points avoiding the leftmost
cylinder $[0, s^{-m})$ under iteration of $g^m$, and hence
$HD(K') \ge \frac{\log s^m-1}{\log s^m} \to 1$ as $m \to \infty$. 
If $t < s^{-2m}$, then $K(t) \supset K'$, because if $k$ is such that $g^k(x) \in [0, s^{-2m})$, then there is
$r \le m$ such that $k+r$ is a multiple of $m$ and
$g^{k+r}(x) \in [0, s^{-m})$.

A similar proof shows that also $HD(K(t) \cap [0,t]) \to 1$. 
Using Proposition~\ref{P:HB}
again, we have $HD(\EE \cap [0,1]) \ge HD K(t) \cap [0,t]) \to 1$
as $t \to 0$, finishing the proof.
\end{proof}

\subsection{Monotonicity of topological entropy (integer slope)}

\begin{proposition}\label{P:monotone-htop}
The topological entropy is increasing (constant, decreasing) on every matching interval $J$
with $\Delta>0$ ($\Delta=0$, $\Delta < 0$) respectively.
\end{proposition}

\begin{proof}
By Corollary~\ref{C:periodic}, periodic points $p \in J$ with period 
$\per(p) \geq \max\{\kappa^+, \kappa^-\}$
will remain periodic but their period decreases by $\Delta$ as the parameter
$\gamma$ (moving from left to right) ``overtakes'' $p$.
There are only finitely many periodic points with $\per(p) < \max\{\kappa^+, \kappa^-\}$.
Since topological entropy is the exponential growth rate of the number 
of $n$-periodic orbits, $\gamma \mapsto h_{top}(Q_\gamma)$ 
increases/decreases/remains unchanged
according to whether $\Delta > 0$,  $\Delta < 0$ or $\Delta = 0$. 
\end{proof}

\subsection{Pseudocenters}\label{sec:pseudo}
 In the previous section we noted that $\EE\subset  [0,s/(s+1)]$; now we shall show that there is a canonical set of labels for the components of  $[0,s/(s+1)] \setminus \EE$ which
 turns out also to be useful to keep track of the matching index.
 
Let $\QS$ denote the set of $s$-adic rationals contained in $(0,1]$.

Let  $u$ be a (finite or infinite)  string composed with the alphabet $\mathcal{A}_s:=\{0,1,...,s-1\}$, 
and let $\check{u}$ be the string obtained by $u$ 
flipping each digit by the involution $\epsilon \mapsto s-1-\epsilon$. For instance, 
in the case $s=2$ if $u=000101$, then  $\check{u}=111010$). Note that if $w$ is an {\bf infinite} string with digits in $\mathcal{A}_s$ and $x=.w$ is the corresponding 
 expansion in base $s$, then  $1-x=.\check{w}$.

\begin{definition}\label{D:thickening}
Let $\xi \in \QS$ and let $w$ denote the shortest base $s$ expansion of {\bf even} length
  of $\xi$ and $v$ denote the shortest  base $s$ expansion of {\bf odd} length of $1-\xi$. We
  define the {\it rational interval} generated by $\xi$ as the
  interval $I_\xi:=(\xi_L, \xi_R)$ containing $\xi$ where 
 the endpoints  are given by
$$
\xi_L:=.\overline{\check{v}v}, \ \ \ \xi_R:=.\overline{w}.
$$
\end{definition}

Let us set $b:=s-1$ and $a:=s-2$;
if $\xi=1-1/s$ then $w=b0$, $v=1$ and
$\xi_L=.\overline{a1}$ while $\xi_R=.\overline{b0}$. This is
(almost) the most degenerate example. In fact, for $\xi\in
\QS\setminus \{1-1/s\}$, one can rephrase the definition of both
$\xi_L, \xi_R$ using the (even) expansion of $\xi$ only. Indeed:

\begin{lemma}\label{L:reloaded}
Let $\xi \in [0,1)$ be an $s$-adic rational with even $s$-adic expansion,
so  $\xi = 0.w = 0.\e_1\e_2\dots \e_{2m-1}\e_{2m}$.
Define $v$ to be the odd $s$-adic expansion of $1-\xi$. Then
$$
v = \begin{cases}
\check{\e}_1...\check{\e}_{2m-2}(\check{\e}_{2m-1}+1) & \text{ if } \e_{2m}=0\\
\check{\e}_1...\check{\e}_{2m-1}(\check{\e}_{2m}+1)0 & \text{ if } \e_{2m} \neq 0.
\end{cases}
$$
where $\check \e = s-1-\e$. Note that $\check w < v$ and  $\check v <  w$. 
\end{lemma}

\begin{proof} This is a straightforward computation.
\end{proof}

\qquad\begin{minipage}[h]{0.9\textwidth}{\footnotesize 
\begin{example} 
We give some examples for $s=2$ in table-form:
$$
\begin{array}{rl|rl|rl}
\xi& & \xi_R& & \xi_L& \\
\hline
& & & & & \\
\frac12 &=\ .10 & \frac23 &=\ .\overline{10} & \frac13 &=\ .\overline{01} \\[2mm]
\frac14 &=\ .01 & \frac13 &=\ .\overline{01} & \frac29 &=\ .\overline{001110} \\[2mm]
\frac{7}{32}&=\ .001110 & \frac29 &=\ .\overline{001110} & \frac{7}{33} &=\ .\overline{0011011001} \\[2mm]
 \frac{3}{16}&=\ .0011 & \frac15 &=\ .\overline{0011} & \frac{2}{11} &=\ .\overline{0010111010} \\[2mm]
\frac{9}{64} & =\ .001001 & \frac17& =\ .\overline{001} & \frac{4334}{16383} & =\ .\overline{00100011101110}\\[2mm]
\frac18&=\ .0010 & \frac{2}{15} &=\ .\overline{0010} & \frac19 &=\ .\overline{000111} 
\end{array}
$$
The penultimate example in this table shows that right endpoint $\xi_R$ can have a minimal 
period shorter than the length of $\xi$.
In all of the above examples, the endpoints belong to the exceptional set, but this need not be the case, 
for instance if 
$\xi=5/16=.0101$ then $\xi_R= .\overline{0101}=5/15=1/3$ but $\xi_L= .\overline{01000110110} \notin \EE$. 
However, $I_{5/16} \subset I_{1/4}$; this follows from a general rule that we shall explain in 
Corollary~\ref{cor:EE} below.
\end{example}}
\end{minipage}

\begin{lemma}\label{L:ixi}
If $\xi \in \QS$ then $I_\xi \cap \EE=\emptyset$. 
Therefore 
$$
\left[0,\frac{s}{s+1}\right]\setminus\EE \supset \bigcup_{\xi \in \QS} I_\xi.
$$
\end{lemma}

\begin{proof}
Let us consider the expanding map $f(x)=sx  \pmod 1$ and the involution $\tau(x):=1-x$; it is easy to 
check that $f$ and $\tau$ commute. Moreover,
\begin{equation}\label{eq:oddm}
g^m= 
\begin{cases}
f^m & \text{ if $m$ is even,} \\
f^m\circ \tau & \text{ if $m$ is odd.}
\end{cases}
\end{equation}
Let us first assume $\gamma\in [\xi,
\xi_R)$. Let $\xi_R:=.\overline{w}$ and let $n$ denote the
length of $w$. Since $n$ is even, $g^n: \ [\xi, \xi_R] \to  [0, \xi_R]$ is a continuous orientation 
preserving expansive map, and $g^n(\xi_R)=\xi_R$.
Therefore $g^n(\gamma)<\gamma$ for all $\gamma \in [\xi,
\xi_R)$, and hence $\gamma \notin \EE$.
If  $\gamma\in (\xi_L, \xi)$, let $m:=|v|$.  Since $m$ is odd,
it follows by equation \eqref{eq:oddm} that 
$g^m:(\xi_L,\xi) \stackrel{\sim}{\longrightarrow}(0,\xi_L)$ is an orientation reversing homeomorphism, hence
$g^m(\gamma)<g^m(\xi_L)=\xi_L<\gamma$ and  
$\gamma \notin \EE$.
\end{proof}

\begin{proposition}\label{P:maximal}
Let $J=(c, d)$ be a connected component of $[0,s/(s+1)]\setminus \EE$.
Then there is a unique $s$-adic $\xi\in J\cap \QS$ of minimal denominator.
Moreover 
\begin{enumerate}
\item[(i)]
 $c=\xi_L$, $d=\xi_R$ (\ie not only $I_\xi \cap \EE=\emptyset$ 
 but $I_\xi$ is maximal with respect to this property, since  $I_\xi=J$);
\item[(ii)]
$g^k(\xi) \notin (0,d)$ for all $k \geq 1$.
 \item[(iii)]
If $\xi'\in \QS\cap (c, d)$ is such that $ g^k(\xi') \notin (0,\xi')$
for all $k \geq 1$, then $\xi'=\xi$.
\end{enumerate}
\end{proposition}

\begin{definition}
If $J$ is a connected component of $[0,s/(s+1)]\setminus \EE$, the  unique $s$-adic 
$\xi\in \QS \cap J$ of minimal denominator will be called the {\bf pseudocenter} of $J$.
We will denote by $\QMAX \subset \QS$  the set of
pseudocenters of components $J$ of $[0,s/(s+1)]\setminus \EE$.
\end{definition}

Pseudocenters provide a convenient way of labeling the connected components 
of $[0,s/(s+1)]\setminus \EE$. Indeed 
as a corollary of Lemma~\ref{L:ixi} and Proposition~\ref{P:maximal} we get
\begin{corollary}\label{cor:EE}
\begin{equation}
[0,\frac{s}{s+1}]\setminus\EE = \bigcup_{\xi \in \QMAX} I_\xi.
\end{equation}
\end{corollary}

\begin{proof}
 Let $h_0:=\max \{k\in \N \ : \ g^k_{|_{[c,d]}} \mbox{ is continuous} \}$. 
We claim that if $0\leq k \leq h_0$ then
\begin{equation}\label{eq:*}
g^k(x) \geq x \ \ \ \forall x \in [c,d].
\end{equation}
Indeed, if $k$ is odd then $g^k_{|_{[c,d]}}$ is a continuous orientation reversing map and
$$x\leq b \leq g^k(b) \leq g^k(x) \ \ \ \forall x\in [c,d].$$
Conversely, if $k$ is even then $g^k_{|_{[c,d]}}$ is a continuous orientation preserving expanding map, so
$g^k(x)- g^k(c) > x-c$ hence $g^k(x)-x> g^k(c)-c\geq 0$, and $g^k(x)>x$.

Let us consider the set $Z:=\{\frac{1}{s}, \frac{2}{s}, ..., \frac{s-1}{s}\}$ of discontinuity points of $g$, and let us point out that  $g^{h_0}([c,d]) \subset (0,1)$ but
 $Z \cap g^{h_0}([c,d])\neq \emptyset $ (by maximality of $h_0$).
 Now, if in $J$ there were more than one $s$-adic rational with minimal denominator, we can find a couple $\xi, \xi' \in \QS $
 such that $g^{h_0}(\xi),g^{h_0}(\xi')\in Z$ and $g^{h_0+1}:(\xi,\xi') \stackrel{\sim}{\longrightarrow}(0,1)$. In particular there exists $p\in J$ such that $g^{h_0+1}(p)=d$, whence $g^k(p)\geq p \ \ \ \forall k\geq 0$
 which is a contradiction.
 Therefore  there is a unique $\xi \in \QS$ if minimal denominator
in $J$. Moreover, $\xi$ is the unique discontinuity of   $g^{h_0+1}_{|_{[c,d]}}$ 
(by minimality of the pseudocenter).
 Thus  $g^{h_0+1}_{|_{[c,\xi[}}$ and $g^{h_0+1}_{|_{]\xi,d]}}$ are continuous as well.
 
Now let us consider $
\xi_L:=.\overline{\check{v}v}, \ \ \ \xi_R:=.\overline{w}$ (where $w,v \in \{0,1\}^*$, $\xi=.w$, $1-\xi=.v$,  $|w|$ even, $|v|$ odd). 

Since, by Lemma~\ref{L:ixi}, $[\xi_L,\xi_R] \subset [c,d]$,
 in order to prove {\bf (i)} it is enough to check that $\xi_L$ and $\xi_R$ both belong to $\EE$, 
 \ie they satisfy
\begin{equation}\label{eq:**}
 g^k(\xi)\geq \xi \ \ \ \forall k \in \N.
\end{equation}
We split the discussion into two cases.

{\bf [$h_0$ is even]} 
then
$$ |w|=h_0+2, \ \ \ |v|=h_0+1.$$
By \eqref{eq:*} $g^k(\xi_L) \geq \xi_L$ for $k\leq  h_0$, and since $g^{h_0+1}(\xi_L)=\xi_L$, 
\eqref{eq:**} holds for $\xi_L$.

Also $\xi_R$ is periodic: $g^{h_0+2}(\xi_R)=\xi_R$. Thus we only have to check that $g^k(\xi_R)\geq \xi_R$ 
for $k\leq h_0+1$.
The range $k\leq h_0$ is covered by \eqref{eq:*}; on the other hand since $g^{h_0+1}$ 
is continuous and orientation reversing on $]\xi, d]$ we get
$$g^{h_0+1}(\xi_R) \geq g^{h_0+1} (b) \geq d \geq \xi_R.$$

{\bf [$h_0$ is odd]} 
then 
$$ |w|=h_0+1, \ \ \ |v|=h_0+2.$$
By \eqref{eq:*} $g^k(\xi_R) \geq \xi_R$ for $k\leq  h_0$, and since $g^{h_0+1}(\xi_R)=\xi_R$, 
\eqref{eq:**} holds for $\xi_R$.

Also $\xi_L$ is periodic: $g^{h_0+2}(\xi_L)=\xi_L$. Thus we only have to check that $g^k(\xi_L)\geq \xi_L$ for $k\leq h_0+1$.
The range $k\leq h_0$ is covered by \eqref{eq:*}; on the other hand since $g^{h_0+1}$ is a 
continuous and orientation preserving map on $[a, \xi[$ we get
$g^{h_0+1}(\xi_L) - g^{h_0+1}(c)\geq \xi_L -a$ hence
$$g^{h_0+1}(\xi_L) - \xi_L \geq  g^{h_0+1}(c) - c \geq 0$$
\ie $g^{h_0+1}(\xi_L) \geq \xi_L$ and we are done.

In order {\bf to prove (ii)} we first point out that $g_k(\xi) \notin (c,d)$ for all $k \geq 1$. 
On the other hand $g^k(\xi)=0$ 
for all $k\geq h_0+1$  while, by equation \eqref{eq:*},
 $g^k(\xi)\geq \xi$ for $k\leq h_0$; thus $g^k(\xi) \notin (0,d)$ for all $k\geq 1$.

Let us prove {\bf  (iii).} Note that if $\xi' \in (\xi_L,\xi)$ then $g^{|v|}(\xi')\in (0,\xi_L)$ thus 
$$ 0 < g^{|v|}(\xi')< \xi_L< \xi'.$$
On the other hand, if $\xi' \in (\xi,\xi_R)$   then 
$\xi_R -\xi'< g^{|w|}(\xi_R)- g^{|w|}(\xi')$ so $$0= g^{|w|}(\xi_R)- \xi_R> g^{|w|}(\xi') -\xi',$$ 
\ie $g(\xi')<\xi'$, and we are done.
\end{proof}

\begin{corollary}\label{C:characterization}
Let $\xi\in \QS \cap (0,s/(s+1))$. Then
$\xi \in \QMAX$ if and only if
$g^k(\xi) \notin (0,\xi) \ \forall k \geq 1$.
\end{corollary}

\begin{proof}
 The implication $\Rightarrow$ is just Proposition~\ref{P:maximal}-(ii). 
On the other hand, if $\xi\in \QS\cap (0,s/(s+1))$ 
 then $\xi$  belongs to some component $(c,d)$ of  $[0,s/(s+1)]\setminus\EE$.  
Thus $g^k(\xi) \notin (0,\xi) \ \forall k \geq 1$ implies, by Proposition~\ref{P:maximal}-(iii),
 that $\xi$ is the pseudocenter of $(c,d)$.
\end{proof}
\subsection{Period doubling}
Another interesting consequence of the above characterization is the following:
\begin{corollary}\label{C:period-doubling}
Let $\xi\in \QMAX \cap (0,s/(s+1))$ and let $I_\xi=(\xi_L, \xi_R)$ with $\xi_L=.\overline{\check{v}v}$. 
Then $\xi':=.\check{v}v \in \QMAX$ as well.
\end{corollary}
In other words, on the left of any matching interval there is an adjacent matching interval, 
hence there is a sequence of adjacent matching intervals. We shall refer to this phenomenon as 
{\em period doubling bifurcation}, in analogy with period
doubling bifurcations in the quadratic family $z \mapsto z^2+c$.
Using Lemma~\ref{L:reloaded} one can easily check that the first few elements of the period doubling 
cascade are as follows:
\begin{equation}\label{eq:period-doubling}
 \begin{array}{l}
\xi_0=.w\\
\xi_1=.\check{v}v\\
\xi_2=.\check{v}\check{w}vw\\
\xi_3=.\check{v}\check{w}v\check{v}vw\check{v}v\\
\xi_4=.\check{v}\check{w}v\check{v}vw\check{v}\check{w}vw\check{v}v\check{v} \check{w} vw 
\end{array} 
\end{equation}

This period doubling phenomenon  is just a particular case of tuning, we shall come back to it later 
on (see Proposition~\ref{P:period-doubling}).
A period-doubling cascade can also be described in terms of a substitution operator.
\begin{lemma}
Consider the substitution
$$
\chi:\begin{array}{ll}
w \mapsto \check v v \qquad &
\check w \mapsto v \check v  \\
v \mapsto vw &
\check v \mapsto \check v  \check w 
\end{array}.
$$
If $\xi$ is a pseudocenter with even $s$-adic expansion $w$, then 
the $s$-adic code of the
pseudocenter of the period doubled matching interval adjacent to
$I_\xi$ is $\chi(w)$.
Continuing this way, we find the $s$-adic codes of the
pseudocenter of the matching interval in the cascade with seed $\xi$.
\end{lemma}

\begin{remark}
This substitution factorizes over the Thue-Morse substitution 
$\chi_{TM}:0 \mapsto 01; \ 1 \mapsto 10$ 
(via the change of symbols $\pi(v) = \pi(\check w) = 0$, $\pi(\check v) = \pi(w) = 1$), 
which in turn factorizes over the period doubling substitution 
$\chi_{PD}:0 \mapsto 11; \ 1 \mapsto 10$.
\end{remark}

\begin{remark}
Denote the length of $\chi^n(w)$ by $l_n$.
Since $w \stackrel{\chi}{\longrightarrow} \check v v
\stackrel{\chi}{\longrightarrow} \check v \check w v w$,
we find the recursive relation $l_{n+2} = l_{n+1} + 2l_n$, which is
solved by $l_n = 2^n \frac{|w|+2|v|}{3} + (-1)^n \frac{2|w|-2|v|}{3}$.
\end{remark}

\begin{proof}[Proof of Corollary \ref{C:period-doubling}]
 By virtue of Corollary~\ref{C:characterization} 
it suffices to check that $g^k(\xi') \notin (0,\xi')$. 
 Let us first point out that, setting $m:=|v|$, we have $g^m(\xi')=\xi$; hence (since $\xi \in \QMAX$) 
 $g^{m+k}(\xi') \notin (0,\xi)$. Thus we just have to check the orbit up to step $m$.
 
 If $k<m$ then $g^k(\xi')=.\sigma$, where $\sigma$ is a suffix of $\check{v}v$ of length $|\sigma|\geq m+1$. 
 On the other hand we know that $g^k(\xi_L) \geq \xi_L$. We claim that in fact $g^k(\xi_L) \geq \xi$.
 This is immediate if $k$ is odd.  For $k$ even let us first remark that it cannot be $g^k(\xi_L) = \xi_L$, 
 because otherwise we would 
 get $\xi_L=.\overline{p\sigma}$ (where $p$ is the prefix of $\check{v}v$ of length $k$), so $\overline{\sigma p}=g^k(\xi_L)= \xi_L=\overline{\sigma p}$, 
would imply $g^k(\xi)=.\sigma  < .\overline{\sigma p}=.\overline{p\sigma }= \xi_L$, contradicting the fact $\xi \in \QMAX$. 
 On the other hand 
 $g^m: ]\xi_L,\xi] \to [0,\xi_L[$ is an order reversing homeomorphism, so 
 if we had that $g^k(\xi_L) \in ]\xi_L,\xi]$
 we would also get $g_{k+m}(\xi_L)=g^m(g^k(\xi_L))< g^m(\xi_L)=\xi_L$, another contradiction.

So we can compare the $s$-adic expansion of $g^k(\xi_L)=.\sigma ... \geq \xi=.w$ 
with that of  $g^k(\xi')=.\sigma $. 
 Since the length of $w$ is $|w|\leq m+1 \leq |\sigma |$ we immediately get $\xi=.w<.\sigma =g^k(\xi')$, 
 and we are done.
\end{proof}

\subsection{Matching index}\label{sec:mi}
Fix an integer $s \ge 2$, and define $g:[0,1] \to [0,1]$ as 
$g(x) = s(1-x) \pmod 1$.
The first return map of $Q_\gamma$
to the interval $[0,1]$ has the form
\begin{equation}\label{eq:R}
R_\gamma(x) = \begin{cases}
g^2(x) = Q_\gamma^{s^2-p+2}(x) & x \in (\gamma,1) \cap [\frac{p-1}{s^2}, \frac{p}{s^2}),\ p = 1, \dots, s^2.\\[2mm]
g(x) = Q_\gamma^{p+1}(x) & x \in (0,\gamma) \cap [\frac{p-1}{s}, \frac{p}{s}),\ p = 1, \dots, s.\\
\end{cases}
\end{equation}

\begin{remark}\label{rem:ab}
In particular, 
if we code the domains of the branches of $g^2$ by blocks $ab$, $a,b \in \{ 0, \dots, s-1\}$, then $p = sa+b+1$, so that
\begin{equation}\label{eq:ab}
R_\gamma(x) = 
\begin{cases}
Q_\gamma^{s^2-sa-b+1}(x) & \text{ if } x \in (\gamma,1) \cap [ab], \\[1mm]
Q_\gamma^{a+2}(x) & \text{ if } x \in (0,\gamma) \cap [ab].
\end{cases}
\end{equation}
\end{remark}

\begin{proof}[Proof of Theorem~\ref{thm:matching_index}]
Let $\xi=.w$ be the pseudocenter of a matching interval with $s$-ary expansion $w=\e_1...\e_{2m}$. 
As usual, we have to distinguish two cases (cf.\ Lemma~\ref{L:reloaded})

{\bf Case 0} If $\e_{2m}=0$ then $v=\check{\e}_1...\check{\e}_{2m-2}(\check{\e}_{2m-1}+1)$ and we will use equation \eqref{eq:ab} to compute $\kappa^\pm$, recalling that, since $|v| = |w|-1$, matching occurs when $\xi^+$ reaches $1$ 
under iteration of $Q_\gamma$ and
$\xi^-$ reaches $1$ for the second time:
$$
\begin{array}{l}
\kappa^+= \sum_{i=1}^{m-1} (s^2-s\e_{2i-1}-\e_{2i}+1) + (s^2-s\e_{2m-1}-\e_{2m}+1),\\[1mm]
\kappa^-= \e_1 +2 + \sum_{i=1}^{m-1} (s^2-s\check{\e}_{2i}-\check{\e}_{2i+1}+1).
\end{array}
$$
Now we compute the difference, keeping in mind that $\check{\e}_{k}=s-1-\e_k$:
\begin{eqnarray*}
\kappa^+-\kappa^- &=& 
\sum_{i=1}^{m-1} (s\check{\e}_{2i}+\check{\e}_{2i+1}- s\e_{2i-1}-\e_{2i}) - s\e_{2m-1}- \e_1 + s^2-1\\[1pt]
&=& (1+s)(s-1)(m-1)- s\sum_{i=1}^{m-1}(\e_{2i}+\e_{2i-1}) \\[1pt]
&& - \sum_{i=1}^{m-1} (\e_{2i+1}+  \e_{2i})  - s \e_{2m-1}- \e_1 + (s+1)(s-1)\\[1pt]
&=& (s+1)\left[(s-1)m - \sum_{k=1}^{2m}\e_{k}\right].
\end{eqnarray*}

{\bf Case 1} If $\e_{2m}\neq 0$ then $v=\check{\e}_1...\check{\e}_{2m-1}(\check{\e}_{2m}+1)0$
and we will use equation \eqref{eq:ab} to compute $\kappa^\pm$, but this time $|v| = |w|+1$ and hence matching occurs when $\xi^+$ reaches $1$ 
 for the second time and $\xi^-$ reaches $1$.
$$
\begin{array}{l}
\kappa^+= \sum_{i=1}^{m-1} (s^2-s\e_{2i-1}-\e_{2i}+1) + (s^2-s\e_{2m-1}-\e_{2m}+1) +1, \\
\kappa^-= \e_1 +2 + \sum_{i=1}^{m-1} (s^2-s\check{\e}_{2i}-\check{\e}_{2i+1}+1) +s^2 -s\check{\e}_{2m} -s +1.
\end{array}
$$
and since $\check{\e}_{k}=s-1-\e_k$ the difference gives:
\begin{eqnarray*}
\kappa^+-\kappa^- &=& \displaystyle  \sum_{i=1}^{m-1} (s\check{\e}_{2i}+\check{\e}_{2i+1}- s\e_{2i-1}-\e_{2i})
 - s\e_{2m-1}-\e_{2m}- \e_1 +s\check{\e}_{2m} + s-1\\[12pt]
&=& (1+s)(s-1)(m-1)- s\sum_{i=1}^{m-1}(\e_{2i}+\e_{2i-1})- \sum_{i=1}^{m-1} (\e_{2i+1}+  \e_{2i})\\[11pt] &&  - s\e_{2m-1}-\e_{2m}- \e_1 +s(s-1)-s\e_{2m} + s-1\\[12pt]
&=& (s+1)\left[(s-1)m - \sum_{k=1}^{2m}\e_{k}\right]
\end{eqnarray*}
In both cases we get the very same expression, and to conclude the proof we only have to check that 
it is equivalent to the formula given in Theorem~\ref{thm:matching_index} 
(which is quite immediate).
\end{proof}

\begin{corollary}
To the left of every maximal matching interval $I_\xi$, there is a maximal neutral matching interval $I_{\xi'}$ obtained from period doubling,
namely with $w' = \check v v$ being the even $s$-adic expansion of $\xi'$.
In particular, there is a cascade of maximal neutral matching interval
to the left of each $I_\xi$.
\end{corollary}

\begin{proof}
It follows immediately from Corollary~\ref{C:period-doubling}   
that $\xi_L = \xi'_R$ for $\xi'$ as given in the statement.
The shape of $w'$ implies that $|w|_a = |w|_{\check a}$ and
hence $\| w \|= 0$. By Theorem~\ref{thm:matching_index}, 
the matching is neutral.
\end{proof}

\subsection{Tuning windows and plateaux}\label{sec:tuning}
Throughout this section $\xi$ is some pseudocenter with even $s$-adic 
expansion $.w$. As usual we shall denote by $.v$  the odd $s$-adic 
expansion of $1-\xi$, and $\xi_R=.\overline{w}$.

\begin{definition}\label{D:tuning-window}
Let  $\xi_T:=.\check{v}\overline{\check{w}}$.
The interval $T_\xi := [\xi_T, \xi_R]$ will be called {\em tuning window} generated by $\xi\in \QMAX$. 
\end{definition}

For instance the rightmost tuning window is $M_s = [\frac{s}{s+1}-\frac{1}{s},\frac{s}{s+1}]$.

We will show that elements in $\EE\cap [\xi_T, \xi_R]$ have $s$-adic expansion that can be easily described. 
Aiming at this, it is very useful first to consider the set
$K(\xi_T)=\{x: g^k(x)\geq \xi_T \ \ \forall k \geq 0\}$; indeed 
it is easily seen that  $\EE \cap [\xi_T,\xi_R]\subset K(\xi_T)\cap [0,\xi_R]$.
  
\begin{theorem}\label{Thm:Komega}
Let $\xi=.w \in \QMAX$, then the following conditions are equivalent:
\begin{enumerate}
\item[(i)]  $ x\in K(\xi_T)\cap [0,\xi_R]$;
\item[(ii)] $x$ can be written as an infinite concatenation $x=.\sigma_1 \sigma_2 \sigma_3...$ where $\sigma_1\in \{w,\check{v}\}$, $\sigma_j \in \{w,v,\check{w},  \check{v}\}$ for all $j\geq 2$, and adjacent blocks must avoid certain patterns, namely:
$$\sigma_j\sigma_{j+1} \notin \{vv,v\check{w}, \check{v}\check{v}, \check{v}w, wv,w\check{w}, \check{w}\check{v}, \check{w}w\}.$$
\end{enumerate}
\end{theorem}

Before going into the proof, let us remark that the $s$-adic expansion of a point $x$ satisfying condition $(ii)$ 
corresponds to an infinite path (starting with $\check{v} $) 
in Figure~\ref{diagramma}. We shall refer to such expansion as {\em admissible expansion} or 
{\em admissible concatenation}.
 
\begin{figure}[h]
\begin{center}
\unitlength=7mm
\begin{picture}(8,4.2)(0,0)
\put(1.7,2){\fbox{$w$}} \put(1.2,1.9){$\circlearrowleft$}
\put(4,0){\fbox{$\check v$}}\put(4,4){\fbox{$v$}}
\put(4.4,0.8){\vector(0,1){2.6}} \put(4.2,3.4){\vector(0,-1){2.6}}
\put(6.2,2){\fbox{$\check w$}}\put(7,1.9){$\circlearrowleft$}
\put(2.6,1.9){\vector(1,-1){1.3}} \put(3.9,3.5){\vector(-1,-1){1.3}}
\put(4.7,0.6){\vector(1,1){1.3}} \put(6,2.3){\vector(-1,1){1.3}}
\end{picture}
\caption{The graph for the codes of pseudocenters in the tuning window $T_\xi$ for $\xi = .w$.}
\label{diagramma}
\end{center}
\end{figure}
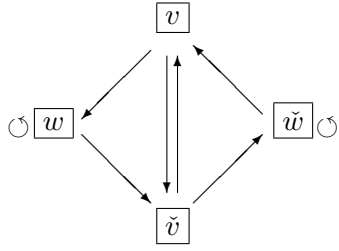
 
 It is easy to check that any admissible expansion has the form
$$ x = .w^{n_0}\check{v}\check{w}^{n_1}vw^{n_2}\check{v}\check{w}^{n_3}v
w^{n_4}\check{v}\check{w}^{n_5}v...$$
where $w^n$ denotes the concatenation of $n$ identical blocks $w$ (possibly none, if $n=0$), and it must be understood that either $n_j\in \Z_+ $ for all $j\geq 0$, or $n_j\in \Z_+$ for $0\leq j <\ell$ and $n_\ell=+\infty$  
(in the case the expansion ends with an infinite tail of $w$ or $\check{w}$).
Note also that in an admissible expansion $n_j$ is the exponent of $w$ when $j$ is even, of $\check{w}$ when $j$ is odd. Moreover if $x$ has an admissible periodic expansion then $x=.\overline{u}$ with $u=w^{n_0}\check{v}\check{w}^{n_1}vw^{n_2}...\check{v}\check{w}^{n_{2\ell-1}}v
w^{n_{2\ell}}$. 

If $x=.\sigma_1,\sigma_2\sigma_3...$  is an admissible expansion then if
$\sigma_j\in\{v,w\} $ then $\sigma_{j+1}\in \{w, \check{v} \}$, while if  $\sigma_j\in\{\check{v},\check{w}\} $ then $\sigma_{j+1}\in \{v, \check{w} \}$.
Since by Lemma~\ref{L:reloaded}  $\check w < v$ and 
$\check v <  w$, this means that the ordering between admissible 
expansions does not depend on the particular $\xi=.w$ which has been chosen. 
For instance, it is immediate to check that $\xi_T=.\check v \overline{\check{w}}$ corresponds to 
the smallest admissible expansion. More precisely, the following result holds:

\begin{lemma}\label{lem:tuning-order}
Let us be given two admissible expansions $$
\begin{array}{l l}
x &= .w^{n_0}\check{v}\check{w}^{n_1}vw^{n_2}\check{v}\check{w}^{n_3}v
w^{n_4}\check{v}\check{w}^{n_5}v...\\
x' &= .w^{n'_0}\check{v}\check{w}^{n'_1}vw^{n'_2}\check{v}\check{w}^{n'_3}v
w^{n'_4}\check{v}\check{w}^{n'_5}v...
\end{array}
$$ 
Assume there is $\bar{k}$ such that $n_{\bar{k}}<n'_{\bar{k}}$ but $n_j=n'_j$ for all $j<\bar{k}$. Then $x<x'$ if and only if $\bar{k}$ is even. 
\end{lemma}

\begin{proof}
If $\bar{k}$ is odd then we get
$$ \begin{array}{ll}
x=& .p \check{w}^{n_{\bar{k}}} v w^{n_{\bar{k}+1}}...\\
x'=& .p \check{w}^{n_{\bar{k}}} \check{w}^{n'_{\bar{k}}-n_{\bar{k}}} v w^{n'_{\bar{k}+1}}...
\end{array}
$$
where $p$ is a common prefix. Looking to the first block where the two expansion are different we read a $v$ for $x$ and a $\check{w}$ for $x'$, since $\check{w}<v$ we can conclude that $x'<x$. An analogous argument works when $\bar{k}$ is even.
\end{proof}

\begin{remark}\label{rk:alt-lex-ord}
Lemma~\ref{lem:tuning-order} shows that, after identifying admissible expansions with the exponents 
$n_0n_1n_2... $ these elements are ordered according to the {\em alternate lexicographic order}.
\end{remark}

\begin{definition}\label{D:ALO}
Let $\mathcal{A}$ be a totally ordered alphabet, we define the {\em alternate lexicographic order} 
$\preceq_{ALO}$ on the space $\mathcal{A}^\N$ of infinite sequences as follows: 
if ${\bf a}=a_1a_2a_3...$ and ${\bf b}=b_1b_2b_3...$ we say that ${\bf a} \preceq_{ALO} {\bf b}$ if
either ${\bf a} = {\bf b}$ or 
$$
\exists k_0 : \ \ a_k=b_k \ \ \forall k<k_0 \quad \mbox{ and } \quad 
\begin{cases}
a_{k_0}<b_{k_0} &\mbox{ if } k_0 \mbox{ is even},\\
a_{k_0}>b_{k_0} &\mbox{ if } k_0 \mbox{ is odd}.
\end{cases} 
$$
\end{definition}
Take the alphabet $\mathcal{A} = \N$ the positive integers.
We can identify an infinite sequence $a_1a_2a_3...$ with the continued fraction 
expansion $[0;a_1,a_2,a_3,...]$; in this case the alternate lexicographic order corresponds 
to the usual order on the reals.

\begin{lemma}\label{L:prefix}
Let $\xi_T:=\check{v}\overline{\check{w}}$ and $K(\xi_T)=\{x: g^k(x)\geq \xi_T \ \ \forall k \geq 0\}$. If $x\in K(\xi_T)\cap [0,\xi_R]$ then $x$ can be written as $x=.wu$ or $x=.\check{v}\check{u}$ for some $u\in \{0,1,...,s-1\}^\N$ such that $u\in K(\xi_T)\cap [0,\xi_R]$.
\end{lemma}
\begin{proof}
If $x\in [\xi, \xi_R]=[.w, .\overline{w}]$ then $x=.wu$; moreover, since $g^{|w|}:[\xi, \xi_R] \to [0, \xi_R]$ is an homeomorphism and $K(\xi_T)$ is $g$-invariant, we see that $g^{|w|}(x)=.u\in K(\xi_T)\cap [0,\xi_R]$.

On the other hand, if $x\in [\xi_T,\xi]$ then
$\tau x\in [1-\xi_T, 1-\xi]=[.v, .v\overline{{w}}]$, hence $\tau (x)=.vu$ \ie $x=.\check{v}\check{u}$. Moreover $g^{|v|}=f^{|v|} \circ \tau$, and since $\tau: [\xi_L,\xi] \to [.v, .v\overline{{w}}]$ and $f^{|v|}: [.v, .v\overline{{w}}]\to [0,\xi_R]$ are both homeomorphisms we get that
$g^{|v|}=f^{|v|} (\tau (x))=.u \in K(\xi_T)\cap [0,\xi_R]$.
\end{proof}

\begin{proof}[Proof of Theorem~\ref{Thm:Komega}]
[{\bf $(i) \Rightarrow (ii)$}]
The fact that $\sigma_1\in \{w,\check{v}\}$ is an immediate consequence of Lemma~\ref{L:prefix}; the same is true for the fact that $\sigma_j\in \{w,v, \check{w},\check{v}\}$ for all $j$. Applying Lemma~\ref{L:prefix} twice we see that the possible initial blocks in the expansion of $x$ are $\check{v}\check{w}, \check{v}v, ww,w\check{v}$. 

Let us prove by induction that the arrows in Figure~\ref{diagramma} represent all possible transitions.
Indeed suppose  $x=.\sigma_1\sigma_2...\sigma_N...$ is a concatenation of blocks  which follows the arrows in the graph up to a certain $N$, then one has one of the following alternative:
\begin{itemize}
\item $\sigma_1...\sigma_{N-1}$ has odd length $p$ and $\sigma_N\in \{\check{w}, v \}$.
\item  $\sigma_1...\sigma_{N-1}$ has even length $p$ and $\sigma_N\in \{\check{v}, w \}$.
\end{itemize}

If we are in the first case, applying $g^p$ with $p$ a suitable even exponent we get
$g^p(x)=.\check{\sigma}_N \check{\sigma}_{N+1}...$ and since $g^p(x)\in K(\xi_T)\cap [0,\xi_R]$ we get
$\sigma_{N}\sigma_{N +1}\in \{vw, v\check{v}, \check{w}\check{w}, \check{w}v  \}$. 
If we are in the second, then
$g^p(x)=.\sigma_N \sigma_{N+1}...$ and $\sigma_N \sigma_{N+1}\in \{ \check{v}\check{w}, \check{v}v, ww,w\check{v}\}$. This proves the admissibility condition holds up to level $N+1$.

[{\bf $(ii) \Rightarrow (i)$}]  Let us recall that $\xi_T=.\check{v}\overline{\check{w}}$ is the smallest 
admissible expansion and let us prove that $\xi_T \in \EE$. By contradiction: if this was 
not the case $\xi_T $ would be inside a matching interval $(\zeta_L, \zeta_R)$, and 
since $\zeta_R \in K(\xi_T)\cap [0,\xi_L]$ we get $\zeta_R=.\overline{\sigma_1...\sigma_\ell}$, 
where $\sigma_1..\sigma_\ell$ is the period of an admissible expansion, and starts with $\check{v}$ 
and ends with $\sigma_\ell\in \{v,w\}$ (because the transition $\sigma_\ell\sigma_1$ must be 
allowed as well). If $\sigma_\ell=v$ then 
$\zeta_L=.\sigma_1...\sigma_{\ell-1}\check{w}\check{\sigma_1}...\check{\sigma_{\ell-1}}w$ 
while if $\sigma_\ell=w$ then  $\zeta_L=.\sigma_1...\sigma_{\ell-1}\check{v}\check{\sigma_1}...\check{\sigma_{\ell-1}}v$. In any case the $s$-adic expansion of $\zeta_L$ 
is an admissible concatenation of blocks starting with $\check{v}$.
Therefore $\zeta_L>\xi_T$ which is a contradiction.

Now, if we consider any infinite admissible concatenation $x=\sigma_1\sigma_2 .... $ we must check that $g^k(x)\geq \xi_T$. If $k=|\sigma_1...\sigma_\ell|$ then there is no problem, since $g^k(x)$ is again an admissible concatenation of blocks. Otherwise we can write $k=|\sigma_1...\sigma_{\ell-1}|+h$ with $1\leq h\leq |\sigma_{\ell}|-1$ and
$g^k(x)=g^{h}(y)$, with $y=.\sigma_\ell\sigma_{\ell+1}...$ where $\sigma_\ell\in \{ w, \check{v} \}$.
If $\sigma_\ell=\check{v}$ then $g^h(y)$ belongs to the interval between 
$g^h(\xi_T)$ and $g^h(\xi)$ (which are both greater than $\xi_T$), therefore $g^h(y)\geq \xi_T$. 
If $\sigma_\ell=w$, the conclusion follows by a similar argument. 
\end{proof}

We recall that  $\EE \cap [\xi_T,\xi_R]\subset K(\xi_T)\cap [0,\xi_R]$, so the previous 
theorem gives a canonical representation for elements of $\EE$  laying in the tuning window. It is then interesting to characterize elements of
$\EE \cap [\xi_T,\xi_R]$, or also elements
 $\QMAX \cap [\xi_T, \xi_R]$ in terms of their period. 
 
\begin{theorem}
Let $x\in K(\xi_T)\cap [0,\xi_L]$ have admissible expansion
$$x = .\check{v}\check{w}^{n_1}vw^{n_2}\check{v}\check{w}^{n_3}v
w^{n_4}\check{v}\check{w}^{n_5}v...$$
Then the following conditions are equivalent
\begin{enumerate}
\item[(i)]  $x\in \EE \cap [\xi_T,\xi_R]$.
\item[(ii)] $n_1n_2n_3 ... \preceq_{ALO} n_k n_{k+1}n_{k+2}...$ for all $k \geq 0$.
\end{enumerate}
\end{theorem}
\begin{proof}
{\bf $[(i)\Rightarrow (ii)]$} follows from Lemma~\ref{lem:tuning-order}.

{\bf $[(i)\Rightarrow (ii)]$}
For any infinite admissible concatenation $x=\sigma_1\sigma_2 .... $ we must check that $g^k(x)\geq x$. If $k=|\sigma_1...\sigma_\ell|$ then there is no problem, since the expansion $g^k(x)$ is again an admissible concatenation of blocks which, by the hypothesis and 
Lemma~\ref{lem:tuning-order}, is no less than $x$. Otherwise we can write $k=|\sigma_1...\sigma_{\ell-1}|+h$ with $1\leq h\leq |\sigma_{\ell}|-1$ and
$g^k(x)=g^{h}(y)$, with $y=.\sigma_\ell\sigma_{\ell+1}...\geq x$ where $\sigma_\ell\in \{ w, \check{v} \}$. If $y\in \EE$ we are done; otherwise $y$ 
lies in a matching interval $(\zeta_L,\zeta_R) \subset [\xi_T, \xi_R]$ and
the same argument as in Proposition~\ref{P:maximal} leads to $g^h(y)\geq \zeta_R\geq x$. 
\end{proof}

\begin{corollary}
Let $\zeta=.z \in  \QMAX \cap [\xi_T, \xi_R]$. Then either $z=w$ or
$$ z=\check{v}\check{w}^{n_1}vw^{n_2}...\check{v}\check{w}^{n_{2\ell-1}}v
w^{n_{2\ell}}$$
where $n_j$ are non-negative integers such that
$n_1n_2...n_{2\ell}$ is minimal among its cyclic permutations in the alternate lexicographic order. 
\end{corollary} 

For instance, given $\zeta \in \QMAX\cap [\xi_T, \xi_R]$ we can well describe its period-doubling sequence of matching intervals in term of admissible expansions.
\begin{proposition}\label{P:period-doubling}
If $\zeta=.\sigma_1...\sigma_\ell\in \QMAX \cap [\xi_T, \xi_R]$ then  $\sigma_\ell \in \{v,w\}$,
$$1-\zeta =.\check{\sigma_1}...\check{\sigma}_{\ell -1}\sigma'_{\ell},  \ \ \mbox{ where } \sigma'_\ell=\left\{
\begin{array}{ll}
 v & \mbox{ if } \sigma_\ell =w\\
 w & \mbox{ if } \sigma_\ell =v
\end{array}
\right.
$$

Therefore the pseudocenter of the matching interval adjacent (on the left) to $I_\zeta$ is 
$\zeta_1:=.\sigma_1...\sigma_{\ell-1}\check{\sigma}'_\ell\check{\sigma}_1...\check{\sigma}_{\ell -1}\sigma'_{\ell}$.
\end{proposition} 

\begin{proof}
 It is enough to prove the formula for the expansion of $1-\zeta$; it is immediate to check that the expression given in the proposition has odd length, and it is also easy to check that adding it to $\zeta$ gives $1$: this because $.\sigma_\ell+.\sigma'_\ell=.w+.v=1$, so in the addition the very last block just generates a carry of 1.
\end{proof}
Using repeatedly this statement we can generate the formulas \eqref{eq:period-doubling}, which describe the first period-doubling cascade.
\medskip 

So far we have shown that, from a combinatorial point of view, all tuning windows look just the same. 
This reflects on the shape of the graph of the entropy,  as we shall see soon. Before stating the next 
result let us introduce the following compact notation: 
$$
\begin{array}{l} {\bf n}:=n_1n_2...n_{2\ell}\in \N_0^*, \\[12pt] 
\llbracket {\bf n}\rrbracket:=\sum_{j=1}^\ell (-1)^j n_j
\end{array} \ \ \ \ \ \  
{\bf n}_w:=\check{v}\check{w}^{n_1}vw^{n_2}...\check{v}\check{w}^{n_{2\ell-1}}v w^{n_{2\ell}}.
$$

\begin{proposition}\label{P:tuning-index}
Let $\zeta:=.\overline{z} \in \EE\cap [\xi_T, \xi_L]$. Then, writing $z= {\bf n}_w$, 
$$ 
\ \ \ \|z\|=\| \underline{n}_w\|=  \llbracket {\bf n}\rrbracket \ \ 
\|w\|
$$
\end{proposition}
\begin{proof}
We shall use the following properties of $\| \|$:
$$ \|u_1 u_2\|=\|u_1\|+\|u_2\|,\ \ \ \ \|\check{u}\|=-\|u\|.$$
Therefore
\begin{eqnarray*} 
\|z\| &=& \|\check{v}\|+n_1\|\check{w}\|+\|v\|+n_2\|w\|+...+\|\check{v}\|+n_{2\ell-1}\|\check{w}\| +\|v\|+ {n_{2\ell}}\|
w\| \\
&=& \|w\|\sum_{j=1}^\ell (-1)^j n_j
\end{eqnarray*}
as required.
\end{proof}

\begin{remark}\label{rem:CT3}
This description allows us to see an unexpected link between the structure of $\EE$ inside a tuning 
window and the set bifurcation set $\EE_{CF}$ for the $\alpha$-continued fractions of Nakada (see \cite{CT3}). Indeed this latter bifurcation set can be characterized by means of the Gauss map $G$ as $$\EE_{CF}=\{x\in [0,1]: G^k(x)\geq x  \ \ \forall k\geq 0\}.$$  
Considering the continued fraction expansion  $x= [0;a_1,a_2,a_3,...]$ one can see 
that $x\in \EE_{CF}$ if and only if the sequence 
$a_1a_2a_3...\preceq_{ALO} a_{k+1}a_{k+2}a_{k+3} \ \forall k \geq 0$.
\ie the string of partial quotients is minimal among its shifted copies with respect to the ALO order. 
The map $\tau_w:\EE_{CF} \to \EE \cap [\xi_T(w), \xi_L(w)]$ defined as 
$$
\tau_w([0;a_1,a_2,...])= \check{v}\check{w}^{a_1-1}vw^{a_2-1}...
$$ 
is an order preserving bijection.
 
Moreover, by virtue of Proposition~\ref{P:tuning-index}, this correspondence reflects on the shape of the entropy: matching intervals of positive, negative or zero index  in the tuning window are intertwined  exactly in the same way as the matching intervals for the $\alpha$-continued fractions. 
\end{remark}

Following \cite{CT3}  let us define the set of untuned parameters  $UT \subset \EE$ as
$$UT:=[0,\frac{s}{s+1}]\setminus \bigcup_{w\in \QMAX} (\xi_T(w), \xi_R(w))$$

\begin{conjecture}
Every element $\zeta \in UT\setminus \{\frac{s}{s+1}  \}$ is accumulated by non-neutral matching intervals.
\end{conjecture}


\subsection{Plateaux}\label{sec:plateaux}


\begin{definition}
A {\em neutral window} for the family $(Q_\gamma)_\gamma$ is a maximal open interval 
$J \subset (0,s/(s+1))$ in parameter space such that
$J$ does not intersect any non-neutral matching interval.
\end{definition}

\qquad\begin{minipage}[h]{0.9\textwidth}{\footnotesize 
\begin{example}\label{ex:M}
The maximal plateau $M_s = [\frac{s}{s+1}-\frac1s, \frac{s}{s+1}] $ from \eqref{eq:M}.
For $\gamma = \frac{s}{s+1}$, 
the map $Q_\gamma$ is continuous and has a Markov partition of $s$ atoms
$[\frac{s}{s+1}-r,\ \frac{s}{s+1}-r+1)$ for
$r = 0, \dots, s-1$.
The transition matrix and characteristic polynomial are
$$
\begin{pmatrix}
0      & 1 & 0 & & & & \\
0      & 0 & 1 & 0 & & & \\
\vdots &   & 0 & 1 & 0 & & \\
  &   &   & \ddots & \ddots & \ddots  & \\
\vdots & & &   & 0 & 1 & 0 \\
0 & & & &  & 0 & 1\\
1 & 1 & \dots &  & \dots & 1 & 1
\end{pmatrix}
\quad \text{ and } \quad
p(\lambda) = \lambda^s - \lambda^{s-1} - \lambda^{s-2} - \dots -1.
$$
Therefore $h_{top}(Q_\gamma)$ is the logarithm of the leading root of $p(\lambda)$.
We know already from Proposition~\ref{prop:mono} that the metric entropy is
$h_\mu(Q_\gamma) = \frac{2\log s}{s+1}$.
Moreover $M_s$ is a maximal plateau since it is accumulated by non-neutral matching intervals 
on the left and the adjacent non-neutral matching interval $[\frac{s}{s+1}, \infty)$ on the right.
\end{example}}
\end{minipage}


The question whether entropy is constant on the entire neutral windows (as the numerics suggest)
or has some devil's staircase behavior is answered by the following:

\begin{theorem}\label{thm:monotone3}
If $J$ is a neutral window for the family $(Q_\gamma)_\gamma$ then both 
the metric  and the topological entropy are constant on the interval $\bar{J}$.
\end{theorem}

\begin{proof}
By Corollary 1 of \cite{KL} we deduce that the dependence of the invariant 
density $d\mu_\gamma$ upon $\gamma$ is  $\eta$-H\"older for any $\eta<1$. 
Consequently also the map $h(\gamma):= h_{\mu_\gamma}(Q_\gamma)$ is $\eta$-H\"older. 
Now let $J\subset [0,s/(s+1)] $ be a neutral window, since the origin is accumulated by non-neutral 
matching intervals it must happen that $\inf J\geq \delta>0$, hence  $HD(J \cap \EE)<1$.

We have thus that by the H\"older property of $h$
$$
HD(h(J \cap \EE))\leq \frac{1}{\eta}HD(J \cap \EE)<1,
$$
where the last inequality above is due to the fact that $h$ is $\eta$-H\"older for any $\eta<1$.
On the other hand since J is a neutral window $h(J \cap \EE)=h(J)$ is an interval, so the fact that $HD(h(J \cap \EE))<1$ implies that $h(J)$ consists of a single point \ie $h$ is constant on $J$.

Now for the topological entropy, let $\gamma_0 \in J$ be arbitrary, and
$U \owns \gamma_0$ is a small neighborhood.
The aim is to show that $\gamma \mapsto h_{top}(Q_\gamma)$
is constant on $U$, so that consequently  $\gamma \mapsto h_{top}(Q_\gamma)$
is constant on the whole tuning window $J$. The idea is that as $\gamma$ moves up through $U$, 
relatively few periodic orbits can change period, so that the exponential 
growth-rate of $n$-periodic points remains unchanged as $\gamma$ varies in $U$.
Although we need to adjust the size of $U$ once in the proof below,
it  holds that $h_{top}(Q_\gamma)$ is locally constant at $\gamma_0$ and since
$\gamma_0$ is arbitrary, $h_{top}(Q_\gamma)$ is constant on $J$.

Clearly, a periodic point $p$ 
undergoes a bifurcation as $\gamma = p$, and one can split the analysis in two:
\begin{enumerate}
 \item $p$ lays in the interior of some matching interval;
 \item $p\in \EE$ and it is the right endpoint of a matching interval (c.f. Lemma \ref{L:pee}).
 
\end{enumerate}

In case (1) the bifurcation has no effect, since per($p$) does not change as $\gamma$ crosses $p$ (c.f.\ Remark \ref{R:PP}). 
Thus, to prove our claim, it suffices to prove that for $U$ sufficiently small, 
the exponential growth rate of $n$-periodic points $p \in U\cap \EE$  is smaller 
than $\inf_{\gamma \in U} h_{top}(Q_\gamma)$.
 
Let $V = [0,v]$ for some $0 < v < \inf U$, and define

\begin{figure}[h]
\begin{center}
\unitlength=3mm
\begin{picture}(28,10)(0,0)
\put(-5,6){$\tilde Q_\gamma(x) = \begin{cases}
1 & \text{ if } x \in V,\\
Q_{\gamma}(x) & \text{ otherwise.}
\end{cases}$}
\thinlines
\put(20,2){\line(1,0){10}} \put(22,0){\line(0,1){10}}
\put(26.7,1.7){\line(0,1){0.6}}\put(26.4,0.8){\tiny $1$}
\put(24.7,1.7){\line(0,1){0.6}}\put(24.4,0.8){\tiny $\gamma$}
\put(21.7,6.7){\line(1,0){0.6}}\put(21,6.5){\tiny $1$}
\thicklines
\put(22,6.7){\line(1,0){1}} \put(22,6.8){\line(1,0){1}} 
\put(22.7,1.7){\line(0,1){0.6}}\put(22.4,0.8){\tiny $v$}
\put(17.2,1.9){\line(1,1){4.8}} \put(24.7,10.3){\line(1,-2){5}} 
\put(17.2,2){\line(1,1){4.8}} \put(24.7,10.4){\line(1,-2){5}} 
\put(23,7.7){\line(1,1){1.7}} \put(23,7.8){\line(1,1){1.7}} 
\end{picture}
\end{center}
\end{figure}
Then $p$ has a periodic orbit for $\tilde Q_\gamma$ if and only if $p$ 
has a periodic orbit for $Q_\gamma$ avoiding $V$.
Lemma~\ref{lem:moststable} shows that 
$h_{top}(\tilde Q_\gamma) < \inf_{\gamma' \in U} h_{top}(Q_{\gamma'})$
provided $U$ is sufficiently small.
Now item (4) follows because 
$h_{top}(\tilde Q_\gamma) = \lim_n \frac1n \log \#\{ n\text{-periodic points of $Q_\gamma$ avoiding } V\}$.
Observe also that if $\orb(p) \cap U = \emptyset$, then $p$
undergoes no bifurcation if $\gamma$ varies in $U$.
This concludes the proof.
\end{proof}

\begin{lemma}\label{lem:moststable}
Let  $\tilde Q_\gamma$ be as in the previous proof.
If $U$ is sufficiently small, then 
$h_{top}(\tilde Q_\gamma) < h_{top}(Q_{\gamma'})$
for every $\gamma, \gamma' \in U$.
\end{lemma}

\begin{proof}
Take $\eps \in (0,1)$ arbitrary and $G^-_\eps = [\gamma-\eps,\gamma^-]$ and 
$G^-_\eps = [\gamma^+, \gamma+\eps]$
(where $\gamma^\pm$ refer to the right/left limit of the discontinuity point $\gamma$).
Write $h = h_{top}(Q_\gamma)$ and let
$K = [s^2(\gamma-1)+1, s(1-\gamma)+1] = [Q^2_\gamma(\gamma^+), Q_\gamma(\gamma^+)]$.
Clearly $Q_\gamma(K) = K$.

Since $s \geq 2$, $Q_\gamma^3(G^-_\eps)$ and $Q_\gamma^2(G^+_\eps)$ 
are intervals of length $4\eps$ and in fact $Q_\gamma^n(G^\pm_\eps)$ contain
intervals of length $4\eps$ for all $n \geq 3$. 
Hence, if $n\geq 3$ is so large that $\gamma \in Q_\gamma^n(G^-_\eps)$,
then $Q_\gamma^n(G^-_\eps) \supset G^-_{2\eps}$ or $G^+_{2\eps}$.
Repeating this argument, we find  that $\orb(G^\pm_\eps) = K$,
and the argument of Lemma~\ref{lem:dense_preperiodic} then
gives that $Q_\gamma$ is transitive\footnote{In the more general family considered in \cite{CM},
transitivity is not guaranteed.}.

Since $Q_\gamma:K \to K$ is transitive, and
some iterate of $Q_\gamma$ is expanding on $K$, 
$Q_\gamma$  supports a unique measure of maximal entropy $\mu$
 and $\mu(V) > 0$, see \cite{Hof}. 

Now $\tilde Q|_K$ is entropy-preservingly semi-conjugate (say via $\psi$)
to a map with slope $\pm e^{\tilde h}$ where 
$\tilde h = h_{top}(\tilde Q_\gamma)$.
Let $\tilde \nu$ be the measure of maximal entropy of this map, and
$\nu = \tilde \nu \circ \psi$.
Then $0 = \nu(\tilde Q(V)) \ge \nu(V)$, because $\nu$ is non-atomic.
It follows that $\supp(\nu) \cap V = \emptyset$, and definitely
$\nu \neq \mu$. However, $\nu$ is not only 
$\tilde Q_\gamma$-invariant, but also $Q_\gamma$-invariant.
Since $\mu$ is the unique measure of maximal entropy of $Q_\gamma$,
it follows that $\tilde h < h$.
Finally, by taking $U$ small we can assume by the
continuity $\gamma \mapsto h_{top}(Q_\gamma)$ 
that $h_{top}(\tilde Q_\gamma) < h_{top}(Q_{\gamma'})$ for 
all $\gamma, \gamma' \in U$.
\end{proof}

The question whether every neutral window is indeed a tuning window
will be discussed as Question (Q2) in the next section.

\section{Numerical evidence and open problems}\label{sec:numerical}
Before speaking about numerical evidence it is good to provide some background information on the objects 
we are interested in, and how we can explore them numerically. 

\subsection{How do we compute?} Let us just recall that there are essentially three different ways of 
computing, namely (a) built-in hardware floating point arithmetic; (b) arbitrary precision arithmetic; 
(c) exact arithmetic (or symbolic) computations. The first method is the default, since it is fast and the 
precision, which is fixed,  is largely adequate for most applications: the {\em double-precision} 
floating-point format available on most modern computers provides about 16 correct decimal digits. 
Method (b) can carry over computations using any (finite) number of correct digits, thus 
going beyond the built-in hardware precision. Finally, method (c) produces an exact result, 
let it be an algebraic number, a binary expansion or a kneading sequence. Method (a) relies 
on the built-in hardware representation of floating point numbers while methods (b) and (c) 
are computationally more expensive and only come with specific libraries or mathematical software
such as Sage, Mathematica or Maple.

When computing with {\em finite precision} (\ie employing methods (a) or (b)) we must be aware of the 
difference between the concepts of {\em precision} and {\em accuracy}: roughly speaking, the term {\em precision} 
indicates the number of digits used to represent floating point numbers, while the {\em accuracy} of a 
computation refers to the number of significant digits of its result. Often accuracy is just slightly 
smaller than precision, 
and yet there are cases where these two quantities differ strongly.
If this happens we say we are facing an {\em ill-conditioned problem}. Overlooking this issue can even 
lead to computations which  produce absurd results because they gain no significant digit at all.

\subsection{What do we compute?}
\begin{description}
\item[Invariant measure and metric entropy of $Q_\gamma$] In principle a numerical approximation of the 
invariant measure $\mu_\gamma$ can be obtained exploiting the fact that the frequency with which a 
typical orbit visits a small interval $I$ is asymptotic to $\mu_\gamma (I)$. These computations also provide 
information about the entropy of $Q_\gamma$.
Indeed, by the Rokhlin formula 
$$
h(Q_\gamma, \mu_\gamma)=\int_R \log|Q_\gamma'(x)|d\mu_\gamma (x)=\log(s)\mu_\gamma([\gamma, +\infty).
$$
Unfortunately this general method is not very effective, and may even fail due to the fact that the computer 
might systematically choose  non-typical points. This failure actually takes place if we use this strategy
and compute with fixed precision the entropy of $Q_\gamma$ when the slope $s=2$: in this case  the problem is 
caused by  the correlation between the slope and the internal binary representation of floating point numbers. 

However, when $\gamma$ belongs to some matching interval, one can use an algorithm which is both more robust 
and much more effective in order to determine the invariant measure (and hence the entropy) of $Q_\gamma$. 
Indeed, we know {\em a priori} that  the invariant density is constant on the complement of the prematching 
set, and computing the invariant density boils down to solving a linear system, an operation which can 
be easily done 
using exact arithmetic. Thus we used this method to compute numerically the metric entropy in 
cases when the matching condition is (or seems to be) dense (see Figure~\ref{fig:entropies} 
and \ref{fig:entro_index}).

Formula~\eqref{eq:interpolation} provides yet another approach to compute the metric entropy on matching intervals: indeed
the  entropy on a matching interval $(a,b)$ only depends on  $h(a)$ and $h(b)$, and these values can 
be computed in a standard way since,  when $\gamma$ equals one of  the endpoints of a matching interval, 
then the map $Q_\gamma$ admits a Markov partition.

\item[Matching intervals] What we discussed just above shows that finding matching intervals for the 
parametric family $(Q_\gamma)_\gamma$  comes with some very precise information about the behavior 
entropy on such parameter values.

By Theorem~\ref{thm:algebraic}, matching can occur in the family $(Q_\gamma)$ if the slope $s$ is an 
algebraic integer, and the quest for matching intervals is indeed an algebraic problem which can be 
dealt with using exact arithmetic in the algebraic number field $\Q[s]$. 

In our numerical computations we adopt the following strategy: we fix a grid of points belonging to  $\Q[s]$ and a safety threshold $N$, then for every $\bar{\gamma}$  belonging to the chosen grid  we check if $Q_{\bar{\gamma}}$ satisfies 
the matching condition with matching exponents $\kappa^\pm \leq N$; if this happens we then determine the endpoints of the matching interval containing $\bar{\gamma}$ by
solving a system of linear equations in $\Q[s]$. We must use a threshold $N$ in order to avoid that 
our algorithm  gets stuck in an excessively long computation (or even infinite - in case $Q_{\bar{\gamma}}$ does 
not satisfy the matching property); and we will have to increase $N$ as we go after smaller and 
smaller matching intervals. 

In the particular case that the slope $s$ is an integer, by  the results of Section~\ref{sec:prevalence} we know  that the  endpoints of every matching intervals are rational and are easily deduced from the pseudocenter. This  provides a much more efficient way of computing matching intervals:
 given an interval $(c,d)$ with $c,d \in \Q \cap \EE$  (for instance we might start 
setting $a:=0$ and $b:=s/(s+1)$) we pick the unique $\xi \in \Q_s \cap (c,d)$ with lowest denominator, this $\xi$  is the pseudocenter a matching 
interval $(\xi_L,\xi_R) \subset (c,d)$ (see Proposition~\ref{P:maximal}); 
 since both $\xi_L$ and $\xi_R$ are bifurcation values 
we can then repeat the same construction to find matching intervals inside $(c,\xi_L)$ and $(\xi_R,d)$ (if these are non-empty intervals). Going on with this bisection algorithm we can reach any fixed matching interval  contained in $(c,d)$  in a finite number of steps.
Let us point out that all these computations are carried out in exact arithmetic (using expansions in base $s$), moreover the 
matching index relative to the matching intervals we find are computed by means of the closed formula 
of Theorem~\ref{thm:matching_index}. 

Let us mention that an analogous strategy works for searching tuning windows.
 
\item[Kneading determinant and topological entropy] We compute the topological entropy through kneading 
invariants. For $s\in\N$ and $\gamma\in \Q$, this quantity can 
be computed with high accuracy: indeed in this case  the map $Q_\gamma$ admits a 
Markov partition and the kneading determinant is a rational function $R_\gamma(t)=p_\gamma(t)/q_\gamma(t)$  
which we compute using exact arithmetic. On the other hand
$h_{top}(Q_\gamma)=-\log(t^*)$ where $t^*$ is the largest positive root of the polynomial $p_\gamma$.
Therefore  we compute the value of the topological entropy with the same accuracy we get for polynomial root finding. 
\end{description}

\subsection{Questions about integer slopes.}
The results of the previous sections provide a detailed description of 
the behavior of the entropy when $s\in \N$, yet some questions remain open. Indeed, even if the numerical 
evidence is quite 
clear we do not have yet a rigorous answer to the following questions:
\begin{enumerate}
\item[(Q1)] Do $h_\mu$ and $h_{top}$ really attain their maximum values on the top plateau $M_s$? 
\item[(Q2)] Does every neutral window coincide with the tuning window generated by some neutral interval?
\end{enumerate}
Let us focus on the latter issue, which is more subtle and admits some partial result.

One can prove that if a neutral window $J$ intersects a non-neutral tuning 
window $(\xi_T, \xi_R)$ then $J\subset (\xi_T, \xi_R)$; thus, by virtue of the canonical homeomorphism 
described in Remark~\ref{rem:CT3}, we can use the results of \cite{CT3} 
to conclude that $J$ coincides with some neutral tuning window. 
In particular, if $w\in \QMAX$ with $\|w\|\neq 0$, then the neutral tuning window of 
endpoints $\xi_L=.\overline{\check{v}v}$ and $\xi'_T=.\check{v}\check{w}\overline{v\check{v}}$ is a 
plateau for the entropy. Indeed, it only contains neutral intervals, it is adjacent to a non neutral 
interval on the right and is accumulated on the left by the non neutral 
intervals $I_{\xi_n}$ with $\xi_n=.\check{v}\check{w}v(\check{v}v)^n$; for instance if $s=2$ and $w=0010$ 
we get that the entropy has a plateau on the interval $[125/1152,1/9]$. 

Question (Q2) admits a positive answer if and only if the following claim is true:

{\bf Claim:} Every neutral tuning window which is {\em primitive} (i.e. it is not properly contained in 
another tuning window) is accumulated both on the right and on the left by non-neutral matching intervals.

This claim can indeed be checked in many particular cases, for instance if $s=2$ and $w=00001111$ we 
have that the tuning window $[\xi_T,\xi_R]$ has endpoints $\xi_T=.000011101\overline{11110000}$ and 
$\xi_R=.\overline{00001111}$ which are accumulated by the matching intervals with 
pseudocenters $.000011101(11110000)^n11110$  and $.0000111011(11100001)^n$, respectively.

\subsection{Irrational slopes}
\begin{figure}[h]
\includegraphics[scale=0.45]{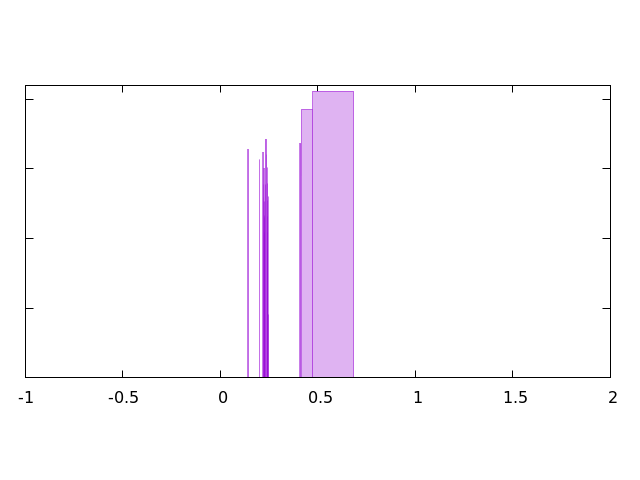}\label{s1pr5}
\caption{The boxes in this picture are built on the matching intervals of $(Q_\gamma)$ for slope $s=\sqrt{5}+1$ 
and their heights are different, depending on the size of the interval. 
The largest matching interval, $(2\sqrt{5}-4, \frac{6}{5}\sqrt{5}-2)$, has matching exponents $(5,6)$ 
and has an adjacent neutral matching interval on the left, $(\frac{24+46\sqrt{5}}{31},  2\sqrt{5}-4)$ 
with matching exponents $(9,9)$. Note that the matching set does not exhaust parameter space; 
in particular there are no matching intervals outside $[0,1]$.}
\end{figure}

As we mentioned in the introduction, the slope $s$ does not need to belong to $\N$ for matching to occur.
Note that  matching may occur for a particular value of $s$ without implying that 
matching is prevalent in the family $(Q_\gamma)_\gamma$.
For instance, for $s=\sqrt{5}-1$ one can find a few matching intervals even if there  is certainly no matching 
interval intersecting the half line $(-\infty,0)$ (this last statement follows easily from 
Remark~\ref{rem:btrans} together with the results of \cite{BCK}).

On the other hand there are several choices for the slope $s$ which seem to lead 
to prevalent matching in the family $(Q_\gamma)_\gamma$; in fact in such cases the entropy has the 
same self-similar features observed when the slope $s\geq 2$ is an integer value. 

One first example of this can be observed 
when the slope $s$ is a quadratic Pisot irrational ($s=(\sqrt{5}+1)/2$, for instance).
Numerical evidence suggests that matching is prevalent, one can also observe the period doubling phenomenon inside the window $[0,s/(s+1)]$, 
and it seems that the bifurcation set has complex fractal structure even inside every plateau of the entropy, 
but complete proofs of all these features are still missing. 
\begin{figure}[h]
\includegraphics[scale=0.4]{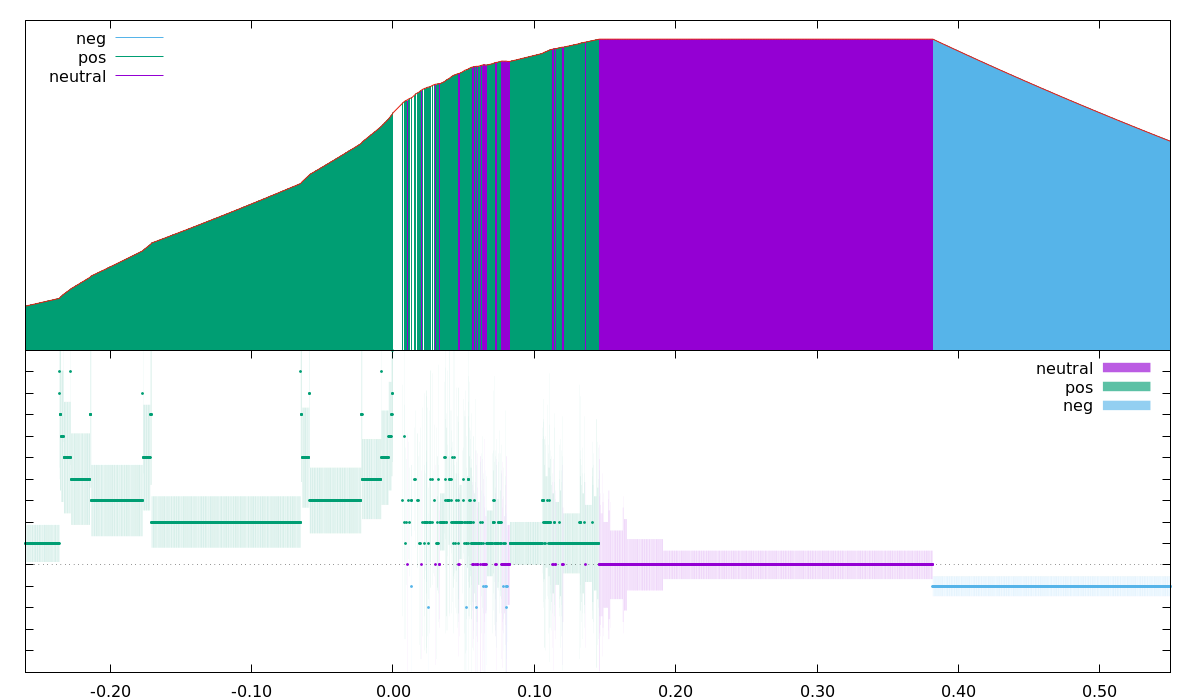}
\caption{$Q_\gamma$ with $s=(\sqrt{5}+1)/2$. Comparing entropy (above) and matching index (bottom); 
in the bottom picture a vertical whisker has been plotted around the value of the matching index, 
the height of the whisker is proportional to the sum of the matching exponents $k_1+k_2$. 
Here we see that the entropy behaves very much like the integer slope case; one can detect 
the first few occurrences of period-doubling, but one can also see that the top plateau 
contains more than a single period-doubling cascade.}
\label{fig:entro_index}
\end{figure}

The plateau which can be seen in Figure~\ref{fig:entro_index} contains many matching intervals; 
the largest being $(\frac{3-\sqrt{5}}{4}, \frac{3-\sqrt{5}}{2})$. 
Numerical evidence suggests the top plateau occurs for  $\gamma \in \frac{7-3\sqrt{5}}{2}, \frac{3-\sqrt{5}}{2}]$). 
\begin{figure}[h]
\includegraphics[scale=0.4]{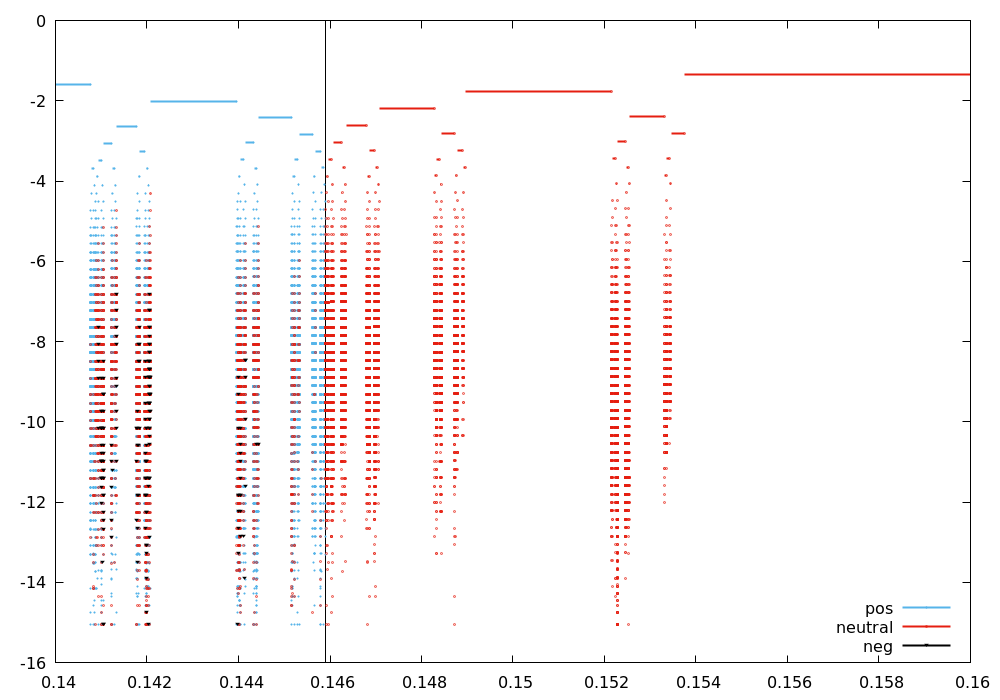}
\caption{$Q_\gamma$ with $s=(\sqrt{5}+1)/2$. Matching intervals plotted at different levels (accordingly to their sizes). It seems that the left endpoint of the first tuning window is $\frac{7-3\sqrt{5}}{2}=0.145898033750315...$. }
\end{figure}

A peculiar feature which marks a difference with the integer slope cases is 
that when the slope is irrational the bifurcation set $\EE$ is not bounded.
Using Remark~\ref{rem:btrans} and the results of \cite{BCK} once again one can prove that $\EE$ has 
in fact a periodic structure outside the bounded interval $[0,s/(s+1)]$.

\begin{figure}[h]
\includegraphics[scale=0.4]{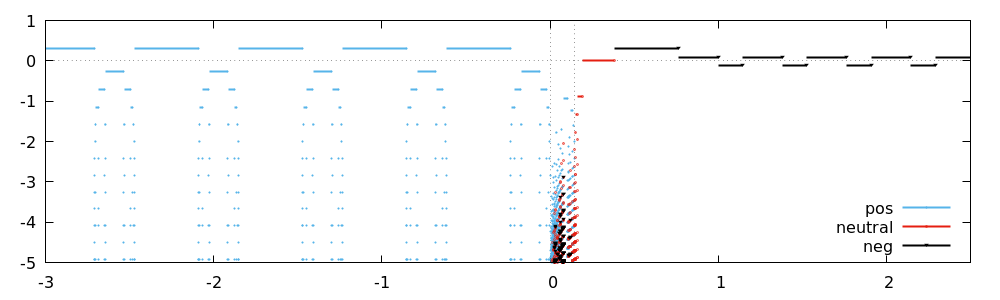}
\caption{$Q_\gamma$ with $s=(\sqrt{5}+1)/2$. Same picture as before, but zooming out it is now evident 
that the periodic structure which extends both on the left and on the right. }
\end{figure}

 With the same techniques can also obtain partial results about prevalence.
 For instance the result of \cite{BCK} imply that for all values $s$ 
 which are quadratic Pisot, $HD( \EE \cap (-\infty, 0))<1$.

\end{document}